\documentclass[11pt]{amsart}

\usepackage[paperwidth=5.96in, paperheight=8.844in,
  left=0.8in, right=0.8in,
  top=1in, bottom=1.5in]{geometry}

\usepackage[hidelinks]{hyperref}
\usepackage{amsmath}
\usepackage{amstext}
\usepackage{amssymb}
\usepackage{amsthm}
\usepackage{enumerate}
\usepackage{bbm}
\usepackage{comment}
\usepackage[USenglish]{babel}
\usepackage[utf8]{inputenc}
\usepackage[T1]{fontenc}

\usepackage{ifthen}
\usepackage{xargs}
\usepackage{xspace}

\usepackage{graphicx}
\usepackage{xcolor}

\usepackage{tikz}
\usetikzlibrary{arrows}

\makeatletter
\def\imod#1{\allowbreak\mkern10mu({\operator@font mod}\,\,#1)}
\makeatother

\theoremstyle{plain}
\newtheorem{theorem}{Theorem}[section]
\theoremstyle{definition}
\newtheorem{definition}[theorem]{Definition}
\newtheorem{corollary}[theorem]{Corollary}
\newtheorem*{cor*}{Corollary}
\newtheorem*{lemma*}{Lemma}
\newtheorem*{theorem*}{Theorem}
\newtheorem*{clm*}{Claim}
\newtheorem*{proposition*}{Proposition}
\newtheorem{lemma}[theorem]{Lemma}
\newtheorem{question}[theorem]{Question}
\newtheorem{proposition}[theorem]{Proposition}
\newtheorem{remark}[theorem]{Remark}

\newtheorem{clm}[theorem]{Claim}

\newtheorem*{acknowledgement*}{Acknowledgement}
\newtheorem{fact}[theorem]{Fact}

\renewcommand{\phi}{\varphi}

\newcommand{\calA}{\mathcal{A}}
\newcommand{\calB}{\mathcal{B}}

\newcommand{\calF}{\mathcal{F}}

\newcommand{\calI}{\ensuremath{\mathcal{I}}}

\newcommand{\calL}{\mathcal{L}}

\newcommand{\calM}{\mathcal{M}}
\newcommand{\calN}{\mathcal{N}}

\newcommand{\calS}{\mathcal{S}}
\newcommand{\calR}{\mathcal{R}}

\newcommand{\card}[1]{|#1|}

\newcommand{\forces}{\Vdash}

\newcommand{\infsubsets}{[\omega]^{\omega}}

\newcommand{\from}{\colon}
\newcommand{\intersection}{\cap}

\newcommand{\union}{\cup}

\newcommand{\concatB}{\mathbin{\rotatebox[origin=c]{90}{\scalebox{.7}{(\kern1ex)}}}}

\newcommand{\suchthat}{\mid}

\newcommandx{\set}[2][2=]{
   \ifthenelse{\equal{#2}{}}{\{#1\}}{\{ #1 \suchthat #2 \}}
}

\newcommand{\ZFC}{\textsf{ZFC}}
\newcommand{\ZF}{\textsf{ZF}}
\newcommand{\ZFminus}{\textsf{ZF}$^-$}
\newcommand{\CH}{\textsf{CH}}

\definecolor{dodger}{rgb}{0.0,0.5,1.0}

\date{}
\begin{document}


\title{Strong Projective Witnesses}


\author{Vera Fischer}
\address{University of Vienna, Institute for Mathematics, Kolingasse 14-16, 1090 Vienna, Austria}
\email{vera.fischer@univie.ac.at}


\author{Julia Millhouse }
\address{University of Vienna, Institute for Mathematics, Kolingasse 14-16, 1090 Vienna, Austria}
\email{julia.marie.millhouse@univie.ac.at}



\makeatletter
\@namedef{subjclassname@2020}{%
  \textup{2020} Mathematics Subject Classification}
\makeatother
\subjclass[2020]{03E35, 03E17, 03E55}


\keywords{projective well-orders; tight mad families; finite logarithmic measures}


\begin{abstract} 
We show Shelah's original creature forcing from 1984 strongly
preserves tight mad families. In particular, answering 
questions of Fischer and Friedman and Friedman and Zdomskyy, 
we show the constellation 
$\aleph_1 = \mathfrak{a} < \mathfrak{s} = \aleph_2$ 
is consistent with the existence of a $\Delta_3^1$ wellorder
of the reals and tight mad families of sizes $\aleph_1, \aleph_2$ which are $\Pi_1^1, \Pi_2^1$-definable, respectively. 
Each of these projective definitions is of minimal possible 
complexity. 

\end{abstract}

\maketitle

\section{Introduction} \label{SEC_Introduction}

In the following we reveal a combinatorial 
property of a rather well-known proper
forcing notion $\mathbb{Q}$ introduced by Shelah in 1984 (see 
Definition \ref{def creature}, and \cite[Definition 6.8]{Shelah84} for the original definition), 
to obtain the consistency of 
$\mathfrak{b} < \mathfrak{s}$ and thus   
establishing the independence of $\mathfrak{b}$ and 
$\mathfrak{s}$. The notion $\mathbb{Q}$ is 
the first instance of so-called creature forcings, which now encompass 
a broad class of posets (see \cite{RScreatures}). Namely, we show in particular 
that $\mathbb{Q}$ satisfies an iterable preservation property (Definition 
\ref{preserve def}) 
introduced in \cite{GHT},
which guarantees the preservation 
of tight mad families under countable support iterations. 

\begin{theorem*}[Theorem \ref{Prop Q strongly preserves tightness}] \label{this is the main theorem}
    Let $\mathbb{Q}$ be the forcing notion of Definition 
    \ref{def creature}, and let $A \in V$ be a tight mad
    family. Let $G$ be $\mathbb{Q}$-generic over
    $V$. Then 
    $(\calA \text{ is a tight mad family})^{V[G]}$.
\end{theorem*}

The poset $\mathbb{Q}$ adds a generic real
unsplit by the ground model reals in a similar manner as 
Mathias forcing, though unlike Mathias forcing, 
it is almost $\omega^\omega$-bounding and so 
countable support iterations of $\mathbb{Q}$ over a model of \textsf{CH} 
preserve the ground model reals unbounded, thus leading to the
consistency of $\aleph_1 = \mathfrak{b} < \mathfrak{s} = \aleph_2$
(see \cite[Theorem 3.1]{Shelah84}).
Since $\mathfrak{b} \leq \mathfrak{a}$, the results of the current paper
give an alternative proof and a  
natural strengthening of the above theorem.
This newly discovered preservation property of Shelah's 
creature poset $\mathbb{Q}$ plays a crucial role in 
obtaining the following interesting result, lying at 
the intersection of descriptive set theory and set theory
of the reals. Using a countable support iteration of 
$S$-proper forcing notions and the preservation of {\emph{nicely definable witnesses}} 
to $\mathfrak{a}$ in such iterations we obtain our main result:
\begin{theorem*}[Theorem \ref{themaintheorem|}]
It is consistent with $\aleph_1 = \mathfrak{a} < \mathfrak{s} = 
\mathfrak{c} = \aleph_2$ that there exists
a $\Delta_3^1$ wellorder of the reals, a coanalytic 
tight mad family of size $\aleph_1$, and a 
$\Pi_2^1$ tight mad family of size $\aleph_2$. 
\end{theorem*}

Classical theorems of 
Mansfield, Mathias \cite{MATHIAS197759}, and Mansfield-Solovay, respectively, imply that these projective definitions are of minimal descriptive complexity, thus optimal. The above 
Theorem will be proved through
a series of intermediary steps, beginning with the
construction of a $\Delta_3^1$ wellorder of the reals
of Fischer and Friedman from \cite{FF2010}, and establishing the following:

\begin{theorem*}[Theorem \ref{delta 13 with a < s}]
It is consistent with $\aleph_1 = \mathfrak{a} < \mathfrak{s} 
= \mathfrak{c} = \aleph_2$ that there exists a $\Delta_3^1$-definable
wellorder of the reals and moreover $\mathfrak{a} = \aleph_1$
is witnessed by a $\Pi^1_1$ tight mad family.
\end{theorem*}

In the recent literature there is an increased interest in the
projective definability of witnesses to cardinal characteristics;
the relevant body of work includes, for example, 
\cite{bergfalk2022projective}, \cite{FFZ11}, 
\cite{FSmed}, \cite{FFK13}, \cite{BK12}, \cite{FFST} or \cite{FSS25}.
We  extend this line of inquiry by 
investigating the projective definability of 
witnesses for all cardinals $\kappa$ belonging to the set
$\mathrm{spec}(\mathfrak{a}) = \set{ \card{\calA}}[\calA \subseteq
\infsubsets, \calA  \text{ is
mad}]$.
Towards this end we examine the $S$-proper
countable support iteration of
Friedman and Zdomskyy \cite{Friedmanzdomskyy} 
used in establishing the consistency of 
$\mathfrak{b} = \mathfrak{c} = \aleph_2$ with 
a $\Pi_2^1$-definable 
tight mad family, which is necessarily of cardinality $\mathfrak{c}=\aleph_2$. 
In particular, answering \cite[Question 18]{Friedmanzdomskyy}, 
we obtain:

\begin{theorem*}[Theorem \ref{Theorem FZ section 2 sizes}] It is consistent that 
$\mathfrak{a} = \aleph_1 < \mathfrak{c} = \aleph_2$, 
and there exists a $\Pi_1^1$-definable tight mad family of
size $\aleph_1$, 
and a $\Pi_2^1$-definable tight mad family of size $\aleph_2$. 
\end{theorem*}


The paper is organized as follows: In Section 2 we introduce the forcing notion $\mathbb{Q}$, 
the notion of a tight mad family and strong preservation of 
tightness, and prove Theorem \ref{Prop Q strongly preserves tightness}. In Section 3
we give an account of the countable support iteration of 
Fischer and Friedman for adjoining  a $\Delta_3^1$ wellorder of the reals, 
and prove Theorem \ref{delta 13 with a < s}. 
Section 4 investigates Friedman-Zdomskyy poset mentioned above 
and proves Theorem \ref{Theorem FZ section 2 sizes}. 
Building on the thus established results, in Section 5 we 
give our main result Theorem \ref{themaintheorem|}. The last section
includes final remarks and open questions.

\section{Creature Forcing}

Define the relation $\leq^*$ of \emph{eventual domination} on the 
collection $\omega^\omega$ by letting $f \leq^* g$ if and only if 
there exists $n \in \omega$ such that for all 
$m \geq n$,
$f(m) \leq g(m)$, and the relation \emph{splits} on 
the collection $\infsubsets$ by letting 
$a $ split $b$ if and only if both 
$b \intersection a$ and $b \setminus a$ are infinite. 
In 1984 Shelah answered a question of Nyikos by showing 
that it is consistent that the minimal size of a family 
$\calF \subseteq \omega^\omega$ which is unbounded with respect
to $\leq^*$ can be strictly less
than the minimal size of a splitting family
$\calS \subseteq \infsubsets$, that is, 
a family $\calS$ with the property that for any infinite 
$b \subseteq \omega$ there is $a \in \calS$ which 
splits $b$. 
 We denote with $\mathfrak{b}$ and $\mathfrak{s}$ these respective 
 cardinalities.
 That $\mathfrak{s} < \mathfrak{b}$ is consistent was already 
 shown in 1980 by Balcar and Simon \cite[Remark 4.7]{BPS80}; the inequality also holds in the Hechler model 
 (see \cite[Section 4]{Shelah84}).
%
%
A family $\calA \subseteq \infsubsets$ is \emph{almost disjoint}
if $a \intersection b$ is finite for every $a,b \in \calA$, 
and such a familly is \emph{maximal almost disjoint (mad)}
if $\calA$ is maximal with respect to inclusion among all
almost disjoint families; equivalently, for every 
$b \in \infsubsets$ there exists $a \in \calA$
such that $a \intersection b$ is infinite. We will always 
consider infinite mad families. The cardinal $\mathfrak{a}$
denotes the minimal size of a mad family; the cardinals 
$\mathfrak{a}$ and $\mathfrak{d}$ are independent.

 One approach to obtaining a model in which 
  $\mathfrak{s} = \aleph_2$
 is with a countable support iteration of Mathias forcing 
 $\mathbb{M}$, which adds a generic real 
 $a \in \infsubsets$ that is not split by any subset of the ground model. However
 Mathias forcing cannot be used to show the consistency
 of $\mathfrak{b} < \mathfrak{s}$,  
 since the generic real $a$ also has the property that its  enumeration function $e_a \from \omega \to a$ is a
dominating real over $V$, thereby witnessing $\omega^\omega \intersection V$ is no longer 
$\leq^*$-unbounded in the $\mathbb{M}$-generic extension. 
In other words, the Mathias reals grow too fast to 
preserve the unboundedness of $\omega^\omega \intersection V$.
To remedy the problem Shelah's original ``creature forcing'' $\mathbb{Q}$ (Definition \ref{def creature}),
slows the growth of the generic
by using a countable sequence of 
\emph{finite logarithmic measures} (see Definiton \ref{def logarithmic measure}) on the second coordinate of a Mathias condition, allowing for careful selection of 
the integers with which one extends a given finite approximation
so to not dominate $\omega^\omega \intersection V$. 
The $\mathbb{Q}$-generic reals are not split by any ground 
model family for similar reasons as with $\mathbb{M}$.
The fact $\mathfrak{b}= \aleph_1$ in a countable 
support iteration of $\mathbb{Q}$ over a model 
of \textsf{CH} was established in \cite{Shelah84} via 
the notion of \emph{almost }$\omega^\omega$\emph{-boundedness},  
see for example \cite{Abraham2010} for further discussion.
Shelah shows $\mathfrak{a}= \aleph_1$ also holds in the above mentioned 
extension by directly constructing a mad family in the ground model
and proving the indestructibility of this specific mad family 
under $\mathbb{Q}$ and its iterations. 
In 1995 Alan Dow \cite{dow95} constructs a mad family 
in a model of $\mathfrak{p} = \mathfrak{c}$ with 
similar properties and shows that it
is indestructible by countable support iterations
of Miller forcing, establishing $\mathfrak{a} = \aleph_1$
in the Miller model. It is interesting to observe that both Shelah's and
Dow's indestructible mad families discussed here are in fact tight, as shown in 
Propositions \ref{construction of shelahs mad family} and \ref{construction of miller indestructible}, respectively.

Since 1984 the theory of forcing preservation has considerably 
developed; see, for example, \cite{ShelahPIP} or \cite{Goldsterntools}.  Likewise has been developed the theory
of \emph{indestructibility} of various witnesses to cardinal 
characteristics, in particular that
of mad families; see \cite{BY05}, \cite{HrusakRationals}, or 
\cite{HrusakFerreira2003}. In 2020, Guzman, Hru\v{s}\'ak, 
and Tellez study the preservation of 
tight mad families, where: 

\begin{definition}\label{def of tight mad family}
An almost disjoint family $\calA \subseteq \infsubsets$ is 
\emph{tight} if for all countable collections $\calB \subseteq 
\calI(\calA)^+$, there exists a single $a \in \calA$ such that 
$a \intersection b$ is infinite for each $b \in \calB$. 
\end{definition}

Above, $\calI(\calA)^+$ denotes the complements of the sets 
belonging to the ideal generated by $\calA$ and the finite
subsets of $\omega$, 
$\calI(\calA) = \set{ b \subseteq \omega}[\exists F \in {[\calA]}^{<\omega} \ (b \subseteq^* \bigcup F)].$

Notice that if an almost disjoint family is tight, then it is 
immediately maximal. This strengthening initially appeared
in the work of Malykhin \cite{Malykhin} in 1989, under 
the name of $\omega$-mad families, and its connection 
to Cohen-indestructibility was established in 
Kurili\'c in 2001 \cite{Kurilic}.
The existence of tight mad families follows from $\mathfrak{b} = \mathfrak{c}$
and a certain parametrized $\lozenge$-principle (see 
\cite{HrusakFerreira2003}). Nonetheless it remains a long standing open question if \ZFC\ proves
the existence of tight mad families.
Recently the following preservation property was 
introduced:

\begin{definition}[{\cite[Definition 7.1]{GHT}}] \label{preserve def}
Let $\calA$ be a tight mad family. We say a proper forcing notion
$\mathbb{P}$ \emph{strongly preserves the tightness of }$\calA$ 
if for every $p \in \mathbb{P}$ and every countable elementary
$\calM \prec H_\theta$, where $\theta$ is a sufficiently large
regular cardinal so that $\mathbb{P}, \calA, p \in \calM$, 
and for every $B \in \mathcal{I}(\calA)$ such that 
$B \intersection Y$ is infinite for all $Y \in 
\mathcal{I}(\calA)^+ \intersection \calM$, 
there exists $q \leq p$ such that 
$q$ is $ (\calM , \mathbb{P})$-generic and 
$q \forces$ ``$ \forall \dot{Z} \in (
\mathcal{I}(\calA)^+ \intersection M[\dot{G}]) ( \card{\dot{Z} 
\intersection B } = \omega)$''. Such a $q$ is called 
an $(\calM, \mathbb{P}, \calA, B)$\emph{-generic condition}.
\end{definition}

It is easy to see that if $\mathbb{P}$ is proper and preserves
the tightness of some $\calA$, then $\calA$ remains tight 
in any $\mathbb{P}$-generic extension. 
Crucially:

\begin{lemma}[{\cite[Corollary 6.5]{GHT}}] \label{csi iterations strongly preserve tightness}
Let $\gamma$ be an ordinal, and let 
 $\mathbb{P} = \langle \mathbb{P}_\alpha, \dot{\mathbb{Q}}_\beta
: \alpha \leq \gamma,  \beta < \gamma \rangle$ be a countable support iteration 
such that for all $\alpha < \gamma$, $\forces_{\mathbb{P}_\alpha} 
$`` $\dot{\mathbb{Q}}_\alpha$ strongly preserves the 
tightness of $\check{\calA}$''. Then $\mathbb{P}$
strongly preserves the tightness of $\calA$.
\end{lemma}

The proof of the above follows the lines of the preservation of properness
under countable support iterations (see \cite{Abraham2010};
\cite{Shelahproperforcing}).
Examples of posets strongly preserving tightness are 
Miller forcing, Miller partition forcing, and Sacks forcing. 
Our  Theorem \ref{Prop Q strongly preserves tightness} adds  
Shelah's original creature forcing to this list, thus
providing an alternative proof of \cite[Theorem 3.1]{Shelah84}
(see Theorem \ref{Theorem a < s}).



An already recurring approach in establishing that a proper forcing 
strongly preserves tightness as in Definition \ref{preserve def}
is to appeal to an
 \emph{outer hull argument}.
  
\begin{definition}
If $\mathbb{P}$ is a forcing notion, $p \in \mathbb{P}$ and 
$\dot{Z}$ is a $\mathbb{P}$-name for a subset of $\omega$, 
the \emph{outer hull of }$\dot{Z}$\emph{ with respect to 
p} is the set 
$W_p:= \set{m \in \omega}[\exists r \leq p \ (r \forces 
m \in \dot{Z})].$

\end{definition}

\begin{fact} [{\cite[Lemma 6.2]{GHT}}] 
\label{outer hull in filter fact}
For an almost disjoint
 family $\calA$ and $\mathbb{P}$ a forcing notion, if $\dot{Z}$ is a $\mathbb{P}$-name for an element of 
$\calI(\calA)^+$, then for any $p \in \mathbb{P}$, 
$W_p \in \calI(\calA)^+$. 
\end{fact}

The above follows from noting that $\dot{Z}$ will be forced
by $p$ to be a subset of $W_p$.
Our main strategy  for establishing strong preservation 
of tightness  will be to
consider a \emph{refined notion} of the outer hull and prove an analogue of Fact \ref{outer hull in filter fact} for these
refinements, see Claim \ref{claim outer hull creature} and Lemma \ref{pure outer hull}.

We proceed with Shelah's and Dow's constructions of tight 
mad families mentioned earlier. 
The following can be found in the proof of
 \cite[Theorem 3.1]{Shelah84} or
 \cite[Theorem 7.1, Chapter VI]{ShelahPIP}.
\begin{proposition}
\label{construction of shelahs mad family}
(\CH) There exists a tight mad family. 
\end{proposition}
\begin{proof} 
Fix an enumeration $\set{ \langle B_n^\alpha : n \in \omega \rangle}[\alpha < 
\omega_1]$ of sequences $\langle B_n^\alpha : n \in \omega 
\rangle$, with $B_n^\alpha \in [\omega]^{<\omega} \setminus 
\set{\emptyset}$ and $B_n^\alpha \intersection B_m^\alpha = 
\emptyset$ for all distinct $m,n \in \omega$. 
By induction on $\alpha < \omega_1$, recursively define a
family $\calA = \set{A_\alpha}[\alpha < \omega_1]$ as follows. 
First, choose a partition $\set{A_n}[n \in \omega]$ of 
$\omega$ into infinite sets. 
For $\alpha \in [\omega, \omega_1)$, 
choose $A_\alpha \subseteq \omega$ so that $A_\alpha$ is 
almost disjoint from $A_\beta$ for each $\beta < \alpha$ and
for each $\beta < \alpha$,
\underline{if} for all $k \in \omega$ and $\alpha_j < \alpha$ for 
$j < k$, for all $m \in \omega$
the set $\set{n \in \omega}[\min(B_n^\beta) > m \wedge \ 
B_n^\beta \intersection \bigcup_{j < k}A_{\alpha_j} = 
\emptyset]$ is infinite, 
\underline{then:} 
\begin{enumerate}
\item there exist infinitely many $n \in \omega$
such that $B_n^\beta \subseteq A_\alpha$, and
\item for all $k \in \omega$ and $\alpha_j \leq \alpha$ 
for $j < k$, 
the set $\set{n \in \omega}[B_n^\beta \intersection 
(\bigcup_{j < k}A_{\alpha_j}) 
= \emptyset]$ is infinite. 
\end{enumerate}

Let $\calA = \set{A_\alpha}[\alpha < 
\omega_1]$. It is straightforward to check that 
$\calA$ is a tight mad family.
\end{proof}


\begin{proposition}[{\cite[Lemma 2.3]{dow95}}] \label{construction of miller indestructible}
Let $\mathfrak{p} = \mathfrak{c}$. Then, 
there is a tight mad family $\calA = \set{A_\alpha}[\alpha < \mathfrak{c}]$ such that for any fixed enumeration
$\set{\langle B_n^\alpha : n \in \omega \rangle }[\alpha
 < \mathfrak{c}]$
of $[[\omega]^{<\omega}]^\omega$, for any $\alpha < \mathfrak{c}$,
\underline{if} there exists $\beta < \alpha$ such that
 \begin{center}$(\ast): $ for all $I \in \mathcal{I}(\bigcup_{\beta < \alpha} A_\beta)$, 
$B_n^\beta \intersection I = \emptyset$ for infinitely many
$n \in \omega$, \end{center}

\noindent
\underline{then} there are infinitely many $n$ such that 
$B_n^\beta \subseteq A_\alpha$. 
\end{proposition}

\begin{proof} 
We define $\calA$ inductively by first letting
 $\set{A_n}[n \in \omega]$ be any partition of 
$\omega$ into infinite sets. Suppose $\{A_\beta\}_{\beta<\alpha}$ have been constructed 
and let 
$\mathcal{I}_\alpha$ denote the ideal
$\mathcal{I}(\bigcup_{\beta < \alpha}A_\beta)$. 
We can assume without loss of generality that $(\ast)$ holds for 
all $\beta < \alpha$.
Next, for each finite partial function 
$s \from \omega \to \alpha$,  define 
$F_s$ to be the set of all 
$x \in [\omega]^{<\omega}$ such that 
$x \intersection \bigcup_{i \in \mathrm{dom}(s)}A_{s(i)} = \emptyset$, and 
 $\forall i \in \mathrm{dom}(s) \ \exists m > \max \mathrm{dom}(s) 
\ (B_m^{s(i)} \subseteq x)$. Then $\calF = \set{ F_s}[ s\from \omega \to \alpha, \text{ s is a finite partial function}] \subseteq [[\omega]^{<\omega}]^\omega$ has the SFIP and is of cardinality 
$\card{\alpha^{<\omega}} < \mathfrak{c} = \mathfrak{p}$,
so there 
exists an infinite $Y \subseteq [\omega]^{< \omega}$ 
such that $Y \subseteq^* F$ for each $F \in \calF$. 
Let $A_\alpha = \bigcup Y$. To see $A_\alpha$ is
as desired,
let $\beta < \alpha$, and let $s \from \omega \to \alpha$ 
be such that $s(0) = \beta$. Then $A_\beta \intersection 
A_\alpha$ is contained in $Y \setminus F_s$, so as the 
latter is finite, so is the former. 

Moreover the set $\set{n \in \omega}[B_n^\beta \subseteq 
A_\alpha]$ is infinite. Note that 
as $Y \intersection F_s$ is infinite,
it is in particular nonempty, 
and if $y \in Y \intersection F_s$ then there is $m > \max \mathrm{dom}(s)$ with 
$B_m^{s(0)} = B_m^\beta \subseteq y \subseteq A_\alpha$. 
Because $\max \mathrm{dom}(s)$ can be taken to be any
$n \in \omega$, as we only require $s(0) = \beta$, 
the integer $m$ above can be taken arbitrarily large. 
It is straightforward to check that $\calA$ is a tight mad family.
\end{proof}

Next we show that $\mathbb{Q}$ strongly preserves tightness. 
Our presentation of the poset follows Abraham's \cite{Abraham2010} (see also \cite{Fischerthesis}).


\begin{definition} \label{def logarithmic measure}
For $s$ a subset of $\omega$, a 
\emph{logarithmic measure on }$s$ is a function  
$h \from [s]^{< \omega} \to \omega$
such that for all $A,B \in [s]^{< \omega}$ and
$\ell \in \omega$, if $h(A \union B) \geq \ell + 1$, 
then either $h(A) \geq \ell$ or $h(B) \geq \ell$. 
A \emph{finite logarithmic measure} is a pair 
$(s,h)$ such that $s \subseteq \omega$ is finite and
$h$ is a logarithmic measure on $s$. 
\end{definition}

\begin{lemma}[{\cite{Shelah84};
\cite[Lemma 2.1.3]{Fischerthesis}}] \label{measures in one partition piece lemma}
If $h$ is a logarithmic measure on $s$ and
$h(A_0 \union \dots \union 
A_{n-1}) > \ell$, then there exists $j < n$ such that
$h(A_j) \geq \ell - j$. 
\end{lemma}

The \emph{level} of a finite logarithmic measure $(s,h)$ is the 
value $h(s)$ and is denoted $\mathrm{level}(h)$. If $A \subset s$ such that $h(A)>0$, 
then $A$ is called $h$-\emph{positive}.

\begin{definition}
Let $P \subseteq [\omega]^{< \omega}$ be an upwards closed collection.
The logarithmic measure $h$ on $[\omega]^{<\omega}$ 
\emph{induced by }$P$ is 
defined inductively on the cardinality of  
$s \in [\omega]^{< \omega}$ as follows: 
\begin{enumerate}
\item $h(e) \geq 0$ for all $e \in [\omega]^{< \omega}$; 
\item $h(e) > 0$ if and only if $e \in P$; 
\item For all $\ell \geq 1$, $h(e) \geq \ell + 1$ if and only
if $\card{e} > 1$ and for all $e_0, e_1 \subseteq e$ 
such that $e = e_0 \union e_1$, 
then $h(e_0) \geq \ell$ or 
$h(e_1) \geq \ell$. 
\end{enumerate}
Then $h(e) = \ell$ if and only if $\ell \in \omega$
is maximal such that $h(e) \geq \ell$. 
\end{definition}


If $h$ is as above and $e$ is such that $h(e) \geq \ell$, 
then $h(a) \geq \ell$ for all sets $a \supseteq e$. 
In the following we will always assume an induced logarithmic 
measure is \emph{non-atomic}, meaning there are
 no $h$-positive singletons. This assumption is necessary for the
 proof of the next lemma, which 
 gives a sufficient condition for an induced 
logarithmic measure to take arbitrarily high values, and
so allowing for the construction of pure extensions with desired properties. The following can be shown by a K\"onig's lemma
argument. 

\begin{lemma}[{\cite[Lemma 4.7]{Abraham2010}}, 
{\cite[Lemma 2.1.9]{Fischerthesis}}]
\label{highvalues}
Let $P \subseteq [\omega]^{<\omega}$ be an upwards closed collection
of nonempty sets, and let $h$ be the induced logarithmic 
measure. Suppose that: 
\begin{center} 
$(\dagger) \ $ for every $n \in \omega$ and 
every  partition 
$\omega = A_0 \union \dots \union A_{n-1}$, 
there is $ i < n$ such that $[A_i]^{< \omega}$ 
contains some $x \in P$.
\end{center} 
Then for every $n,k \in \omega$, and finite partition of
$\omega$ into sets $A_0 \union \dots \union A_{n-1}$, there exists 
$i < n $ and $x \subseteq A_i$ such that $h(x) \geq k$.  
\end{lemma} 

We now define our main forcing of interest.

\begin{definition}[{\cite[Definition 2.8]{Shelah84}}, 
{\cite[Definition 1.3.5]{Fischerthesis}}]
  \label{def creature}
  Let $\mathbb{Q}$ be the partial order
   consisting of pairs $p = (u, T)$ 
  such that $u \subseteq \omega$ is finite and $T$ is a sequence
  $T = \langle t_i : i \in \omega \rangle$, where for all 
  $i \in \omega$, $t_i$ is a pair
  $t_i = (s_i, h_i)$, where $h_i$ is a finite logarithmic measure
  on $s_i$, and such that: 
  \begin{enumerate}
  \item $\max(u) < \min(s_0)$;
  \item $\max (s_i) < \min(s_{i+1})$ for all $i \in \omega$;
  \item $\langle h(s_i) : i \in \omega \rangle$ is unbounded 
  and strictly increasing.
  \end{enumerate}

  For $T$ as above,
  let $\mathrm{int}(t_i) = s_i$, and $\mathrm{int}(T) = 
  \bigcup_{i \in \omega} \mathrm{int}(t_i)$. When $e \subseteq 
  \mathrm{int}(t_i)$ is such that $h_i(e) > 0$, we say that
 $e$ is $t_i$\emph{-positive}. 
  For conditions $(u_0, T_0), (u_1,T_1) \in \mathbb{Q}$, where
  $T_j = \langle t_i^j : i \in \omega \rangle$, 
  $t_i^j = (s_i^j, h_i^j)$ for 
  $j < 2$, define $(u_1, T_1) \leq (u_0, T_0)$ if and only if:
 
  \begin{enumerate}\addtocounter{enumi}{4}
    \item $u_1$ end-extends $u_0$ and $u_1 \setminus u_0 \subseteq
    \mathrm{int}(T_0)$; 
    \item $\mathrm{int}(T_1) \subseteq \mathrm{int}(T_0)$ and
    there is a sequence $\langle B_i : i \in \omega \rangle$
    of finite subsets of $\omega$ such that $\max(B_i) < \min(B_{i+1})$ 
    and for each $i \in \omega$, $s_i^1 \subseteq \bigcup_{j \in 
    B_i} s_j^0$;
    \item For all $i \in \omega $ and $e \subseteq s_i^1$, if 
    $h_i^1(e) > 0$ then there exists $j \in B_i$ such that 
    $h_j^0(e \intersection s_j^0) > 0$. 
  \end{enumerate}
  In the case $u_1=u_0$, call $(u_1,T_1)$ a \emph{pure extension}
  of $(u_0, T_0)$.
\end{definition}

If $(\emptyset,T)$ is a 
condition in $\mathbb{Q}$, we will identify $(\emptyset, T)$
and $T$, and write simply $T \in \mathbb{Q}$. 
For $T = \langle t_i : i \in \omega \rangle \in \mathbb{Q}$ and 
$k \in \omega$, 
let $$i_T(k) = \min \set{i \in \omega}[k < \min(\mathrm{int}(t_i))]$$
and let $T \setminus k = \langle t_i : i \geq i_T(k) \rangle$. 
Then $T \setminus k \in \mathbb{Q}$ and $T \setminus k \leq T$.
Similarly if $u \subseteq \omega$ is finite, 
$T \setminus \max(u) \in \mathbb{Q}$. 
By a slight abuse of notation, we will understand by 
$(u,T \setminus u)$ to mean
 the condition $
(u,T \setminus \max(u))$ in the case 
$\max(u) \geq \min(\mathrm{int}(t_0))$. 
That $\mathbb{Q}$ satisfies Axiom A (see \cite[Definition 2.3]{Abraham2010})
and hence is proper can 
be established by the following.  
\begin{definition} 
  For $n \in \omega$ and $(u_0,T_0), (u_1, T_1) \in \mathbb{Q}$,
  let $\leq_0$ be the usual partial order on 
  $\mathbb{Q}$. Let $T_j = \langle t_i^j : i \in \omega 
  \rangle$ for $j < 2$. Define
\begin{enumerate}
\item   $(u_1, T_1) \leq_1 (u_0, T_0)$
  iff $(u_1,T_1) \leq_0 (u_0, T_0)$ and $u_1 = u_0$;
\item for $n \geq 1$ let
  $(u_1, T_1) \leq_{n+1} (u_0, T_0)$ iff $(u_1, T_1) \leq_1 (u_0, T_0)$ and  $t_i^1 = t_i^0$ for all $i<n$.
\end{enumerate}
 In particular, $(u_1, T_1) \leq_1 (u_0, T_0)$ if and only if
  $(u_1, T_1)$ is a pure extension of $(u_0, T_0)$. 
\end{definition}

  \begin{definition} \label{def fusion for Q}
  Given a sequence $\langle p_i : i \in \omega \rangle \subseteq 
  \mathbb{Q}$, $p_i= (u, T_i)$, $T_i = \langle t_j^i : j
  \in \omega  \rangle$  such that $p_{i+1} \leq_{i+1} p_i$ for all
  $i \in \omega$, define the 
  \emph{fusion} of $\langle p_i : i \in \omega \rangle $
  to be the condition 
  $q:=(u, \langle t_j : j \in \omega \rangle)$ such that 
  $t_j:= t_j^{j+1}$ for all $j  \in \omega$. 
\end{definition} 

If $q$ is the fusion of $\langle p_i : i \in \omega \rangle$,
then $q \leq_{i+1} p_i$ for all $i \in \omega$.
The following notion is crucial both for proving that
$\mathbb{Q}$ is proper and for our preservation result.

\begin{definition}
For $(u,T) \in \mathbb{Q}$, with $T = \langle t_i : i \in 
\omega \rangle$, and $D$ an open dense subset of $\mathbb{Q}$, we say 
$(u,T)$ is \emph{preprocessed for} $D$ \emph{and} $k \in \omega$ 
if for every $v \subseteq k$ such that 
$v$ end-extends $u$, if $(v, \langle t_j : j \geq k \rangle)$
has a pure extension in $D$, then already 
$(v, \langle t_j : j \geq k \rangle) \in D$. 

\end{definition}

Note
that if $(u,T) \in \mathbb{Q}$ is preprocessed for 
$D$ and $k$, then any extension of $(u,T)$ is also preprocessed
for $D$ and $k$. 

\begin{lemma}[{\cite[Lemma 1.3.9]{Fischerthesis}}]
For every open dense subset $D \subseteq \mathbb{Q}$, every 
$k \in \omega$,
 and every $p \in \mathbb{Q}$, 
there exists $q \in \mathbb{Q}$ such that $q \leq_{k+1} p$ 
and $q$ is preprocessed for $D$ and $k$. 

\end{lemma}

As a consequence of the existence of fusion for $\mathbb{Q}$:

\begin{lemma} \label{preprocessedforD}
 For every open dense $D \subseteq \mathbb{Q}$
and every $p \in \mathbb{Q}$ there exists a pure extension 
$q \leq p$ such that $q$ is preprocessed for $D$ and every 
$k \in \omega$. 

\end{lemma}

Let $\dot{C}$ be a $\mathbb{Q}$-name for a subset of 
$\omega$, and let $j \in \omega$. Let $\dot{C}(j)$ denote 
the name for the $j$-th element of $\dot{C}$. 
 We say a condition $p$ \emph{decides} $\dot{C}(j)$
 if there exists $\ell \in \omega$ such that 
 $p \forces \dot{C}(j) = \check{\ell}$. 
 For such $\dot{C}$ and $j \in \omega$, let
$E_{C(j)} = \set{p \in \mathbb{Q}}[p \text{ decides } 
\dot{C}(j)].$ 
When $p\in \mathbb{Q}$ forces that 
$\dot{C}$ is infinite, the set $E_{C(j)}$ is open dense below $p$
in 
$\mathbb{Q}$.


\begin{lemma}\label{stepA}
Let $T\in\mathbb{Q}$, let  
$\dot{C}$ be a $\mathbb{Q}$-name for an infinite subset of 
$\omega$ and let  $n,j \in \omega$. Let $v \subseteq n$. 
Then there is $R = \langle r_i : i \in \omega \rangle \in \mathbb{Q}$ such that 
$R \leq T$ and 
 for all $i \in \omega$ and $r_i$-positive 
$s \subseteq \mathrm{int}(r_i)$, 
there is $w \subseteq s$ such that 
$(v \union w,  R \setminus \max(s))$ decides
$\dot{C}(j)$. 
\end{lemma}

\begin{proof}
Let $T = \langle t_i : i \in \omega \rangle$ with 
$t_i = (s_i, h_i)$.
By Lemma \ref{preprocessedforD} we can
 suppose that $T$ is preprocessed for 
$E_{C(j)}$ and every $k \in \omega$. 
Let $\mathcal{P}_v(C(j))$ denote those 
$ x \in [\mathrm{int}(T)]^{<\omega}$ such that: 
\begin{enumerate}
\item For some $k \in \omega$, $x \intersection \mathrm{int}(t_k)$
is $t_{k}$-positive; 
\item There exists $w \subseteq x$ such that 
$(v \union w, \langle t_i : i > \max(x) \rangle)$ 
decides $\dot{C}(j)$. 
\end{enumerate}
Note that $\mathcal{P}_v(C(j))$ is upwards closed. 
Let $h \from [\omega]^{<\omega} \to \omega$ be the logarithmic measure
induced by $\mathcal{P}_v(C(j))$. We will show that $h$
takes arbitrarily high values by establishing 
$(\dagger)$ of Lemma \ref{highvalues}. 

Fix $M \in \omega$ and a finite partition $\omega = A_0 \union \dots A_{M-1}$. 
First find $T' \leq T$ such that 
$\mathrm{int}(T') \subseteq A_N$ for some 
$N< M$; such an extension exists by K\"onig's lemma 
argument (see, for example, \cite{Fischerthesis}).
Since $(v, T' \setminus v) \in \mathbb{Q}$, there 
is $w \subseteq \mathrm{int}(T' \setminus v)$ and 
$R$ such that 
$(v \union w, R)$ is a condition in $\mathbb{Q}$ extending
$(v, T' \setminus v)$ and $(v \union w, R)$ decides the 
value of $\dot{C}(j)$. 
Since $w \subseteq \mathrm{int}(T' \setminus v)$ is finite, 
by definition of the extension relation  
there are $m_0, m_1$ such that 
$w \subseteq \bigcup_{m \in [m_0, m_1]} \mathrm{int}(t_m')$. 
Let $x:= \bigcup_{m \in [m_0, m_1]} \mathrm{int}(t_m')$ and 
note $x \in [A_N]^{< \omega}$. 
As $T' \setminus v \in \mathbb{Q}$ we may assume there is 
at least one $m \in [m_0, m_1]$ such that 
$\mathrm{int}(t_m')$ is $t_m'$-positive. Then 
$T' \setminus v \leq T \setminus v$ and so there is 
$k \in \omega$ such that 
$h_k(\mathrm{int}(t_m') \intersection \mathrm{int}(t_k)) = 
h_k(x \intersection \mathrm{int}(t_k)) > 0$. 
Therefore $x \subseteq \mathrm{int}(T') \subseteq A_N$
satisfies (1) in the definition of 
$\mathcal{P}_v(C(j))$. 

We can also show that (2) holds. We have 
that $w \subseteq x$ is such that 
$(v \union w, R) \in E_{C(j)}$, but 
$T$ was preprocessed for $E_{C(j)}$ and $\max w$, 
and $(v \union w, T \setminus w)$ has a pure extension into 
$E_{C(j)}$ so already 
$(v \union w, T \setminus w) \in E_{C(j)}$. 
Altogether we have found $x \in \mathcal{P}_v(C(j)) \intersection 
{[}A_N{]}^{< \omega}$, verifying 
$(\dagger)$. 



We can now define $R = \langle r_n : n \in \omega \rangle$, 
where $r_n = (x_n, g_n)$, inductively as follows. 
Clearly $\mathcal{P}_v(C(j))$ is nonempty, as we just showed 
above, so pick $x_0 \in \mathcal{P}_v(C(j))$, 
and let $g_0:= h \restriction \mathcal{P}(x_0)$. 
Assuming $r_i = (x_i, g_i)$ defined for all $i \leq n$ 
so that $\max(x_i)< \min(x_{i+1})$ and $g_i(x_i)< g_{i+1}(x_{i+1})$ 
for $i < n$, since $h$ takes arbitrarily high values there
is $x_{n+1} \in \mathcal{P}_v(C(j))$ with 
$h(x_{n+1}) > h(x_n)$. We can assume 
$\max(x_n) < \min (x_{n+1})$, since otherwise $h$ is bounded. 
Define $g_{n+1} := h \restriction \mathcal{P}(x_{n+1})$. This completes the definition of $R$. 

Then $R \in \mathbb{Q}$, and it is routine to 
check $R$ extends $T$.
To see $R$ is as desired, if $s \subseteq x_i$ is 
$r_i$-positive, by (2) of the
definition of $\mathcal{P}_v(C(j))$ there is $w \subseteq s$
such that $( v \union w, \langle t_i : i > \max(s) \rangle)$ decides 
$C(j)$, but $(v \union w, \langle r_i : i > \max(s) \rangle ) \leq 
( v \union w, \langle t_i : i > \max(s) \rangle)$ and so makes the same 
decision. 
\end{proof}

\begin{remark} \label{stepA preserved under extension}
If a condition $R$ in $\mathbb{Q}$ has the property as in the 
conclusion of the above lemma, then any further extension retains
this same property.
\end{remark}

To obtain the next lemma consider each 
$v \in \mathcal{P}(n)$ and repeatedly apply Lemma 
\ref{stepA} and Remark \ref{stepA preserved under extension}
(see e.g.  \cite{Fischerthesis}).

\begin{lemma}\label{stepB}
For any $T \in \mathbb{Q}$, $n, j \in \omega$ and $\dot{C}$ a 
$\mathbb{Q}$-name for an infinite subset of $\omega$, there 
exists $R = \langle r_i : i \in \omega \rangle \in \mathbb{Q}$
such that $R \leq T$ and for all $v \subseteq n$, for all 
$i \in \omega$ and $r_i$-positive 
$s \subseteq \mathrm{int}(r_i)$, there is 
$w \subseteq s$ such that
$(v \union w, R \setminus s)$ decides $\dot{C}(j)$. 
\end{lemma}

\begin{corollary} \label{n+1 extension of stepB}
For any $(u, T) \in \mathbb{Q}$, any $n, j \in \omega$, and
any $\dot{C}$ a $\mathbb{Q}$-name for an infinite subset of 
$\omega$, 
there exists $(u, R) \leq_{n+1} (u,T)$ such that for all 
$v \subseteq n $, for all 
$i \geq n$ and for every $s \subseteq \mathrm{int}(r_i)$ 
which is $r_i$-positive, there exists 
$w \subseteq s$ such that 
$(v \union w, R \setminus s)$ decides the value of 
$\dot{C}(j)$. 
\end{corollary}



The central result of this section is 
the following.

\begin{theorem} \label{Prop Q strongly preserves tightness}
Let $\calA$ be a tight mad family. For every $p \in \mathbb{Q}$ and
$M \prec H_\theta$ countable elementary submodel, where 
$\theta$ is sufficiently large, containing 
$\mathbb{Q}, p, \calA$, and every $B \in \mathcal{I}(\calA)$ such that 
$\card{B \intersection Y} = \aleph_0$ for all $Y \in \mathcal{I}
(\calA)^+ \intersection M$, and $\dot{Z}$ a $\mathbb{Q}$-name for an 
element of $\mathcal{I}(\calA)^+$ in $M$, there exists 
a pure extension 
$q \leq p$ such that $q$ is $(M, \mathbb{Q}, \calA, B)$-generic. 
\end{theorem}

\begin{proof}
Fix $\theta, M$ and $B$ as above, and let 
$(u, T_0)$ be a condition in $\mathbb{Q} \intersection 
M$.
Let $\set{D_n}[n \in \omega]$ enumerate all open dense subsets of 
$\mathbb{Q}$ in $M$, and let
 $\set{\dot{Z}_n}[n \in \omega]$ enumerate all 
$\mathbb{Q}$-names for subsets of $\omega$ in $M$ which are 
forced to be in $\mathcal{I}(\calA)^+$ such that each name appears
infinitely often. We inductively 
define a sequence $\langle q_n : n \in \omega \rangle$
of conditions in $\mathbb{Q} \intersection M$, where 
$q_0 = (u,T_0)$, $q_n = (u, T_n)$ with $T_n = 
\langle t_i^n : i \in \omega \rangle$, such that the following 
are satisfied: 

\begin{enumerate}
\item For all $n \in \omega$, $q_{n+1} \leq_{n+1} q_n$; 
\item $q_{n+1}$ is preprocessed for $D_n$ and every 
$k \in \omega$;
\item For all $v \subseteq n $,
for all $i \geq n$ and $s \subseteq \mathrm{int}
(t_{i}^{n+1})$ which is $t_i^{n+1}$-positive, if $v$
is an end-extension of $u$, 
then for some $w \subseteq s$, $((v \union w), \langle t_j^{n+1} :  j > 
\max(s) \rangle ) \forces (\dot{Z}_n \intersection B) \setminus n \neq 
\emptyset$. 
\end{enumerate}

Suppose $q_n$ has been constructed. To define 
$q_{n+1}$ consider the condition $(u, \langle t_i^n : i \geq n \rangle
 ) \leq q_n$. By Lemma \ref{preprocessedforD}, there is a pure extension 
$(u, \langle t_i' : i \geq n \rangle ) \leq (u, \langle t_i^n : i \geq n 
\rangle)$ which is preprocessed for $D_n$ and every $k \in \omega$. 
Let $q_n^0 = (u, \langle t_i^{n,0} : i \in \omega \rangle)$, 
where for $i < n$, $t_i^{n,0} = t_i^n$, and $t_i^{n,0}= t_i'$ 
for $i \geq n$. Then $q_n^0 \leq_{n+1} q_n$ and 
$q_n^0$ satisfies (2). 
Next, consider the set 
\begin{align*}
W_{n+1}  = \{ m \in \omega \mid & 
\exists  r = (u, \langle t_i' : i \in \omega \rangle )
\leq_{n+1} q_n^0 \text{ satisfying: }\\
& \forall v \subseteq n\  \forall i \geq n \ 
\forall  s \subseteq \mathrm{int}(t_i')\ 
[s \text{ is } t_i' \text{-positive} \Rightarrow  \\ &
\exists w \subseteq s (v \union w, \langle t_i' : i > \max(s) \rangle)
\forces m \in \dot{Z}_n ]\}
\end{align*}

\begin{clm} \label{claim outer hull creature}
$W_{n+1} \in \mathcal{I}(\calA)^+ \intersection M$.
\end{clm}

\begin{proof}
We have that $W_{n+1} \in M $ since it is definable from
the forcing relation and from 
$\mathbb{Q}$, $q_n^0$, and $\dot{Z}_n$, which are all assumed 
to be elements of $M$. 
To see $W_{n+1} \not\in \calI(\calA)$, let $F$ be a finite subset of $\calA$. We will show that $W_{n+1} \setminus \bigcup F$ is infinite. 
Since $\forces_{\mathbb{Q}} \dot{Z}_n \in \calI(\calA)^+$, 
 in particular 
 $q_n^0$ forces that $\dot{Z}_n \setminus \bigcup F$
  is infinite.
 Let $\dot{C}$ be a $\mathbb{Q}$-name for the set 
 $\dot{Z}_n \setminus \bigcup F$. 

 By Corollary \ref{n+1 extension of stepB}, for 
 all $k \in \omega$ there is 
 $q^j \leq_{n+1} q_n^0$, where
 $q^j = (u, R_j)$ and $R_j = \langle r_i^j : i \in \omega \rangle $ such that 
 for all $v \subseteq n$, for all $i \geq n$ and 
 $r_i^j$-positive $s \subseteq \mathrm{int}(r_i^j)$, there is 
 $w \subseteq s$ such that 
 $(v \union w, R_j \setminus s)$ decides 
 $\dot{C}(j)$ and $j>k$. So there exists 
 $m_j \in \omega$ such that 
 $(v \union w, R_j \setminus s) \forces \dot{C}(j) = 
 \check{m}_j.$
 Note that if $m_j = \dot{C}(j)$, then $m_j \geq j > k$. 
 Therefore for all $k \in \omega$ there exists
 $m_j > k$ such that $m_j \in W_{n+1}$, witnessed by $q^j$ 
 as above, and moreover $m \not\in \bigcup F$ since 
 $q^j \forces \check{m}_j \not \in \bigcup F$. 
 Thus $W_{n+1} \setminus \bigcup F$ is infinite, and
 as $F \in [\calA]^{< \omega}$ was arbitrary this proves the claim. 
\end{proof}

\noindent
By assumption on the set $B$, there exists $m_{n+1} \in 
W_{n+1} \intersection B$ such that $m > n$. 
Let $r = (u, \langle t_i' : i \in \omega \rangle ) \leq_{n+1} 
q_n^0$ be given by $m_{n+1} \in W_{n+1}$, 
and define $q_{n+1}= 
(u, \langle t_i^{n+1} : i \in \omega \rangle)$ 
such that $t_i^{n+1} = t_i'$ for all $i \in \omega$. 
Since $r \leq_{n+1} q_n^0 \leq_{n+1} q_n$, we have that 
$q_{n+1} \leq_{n+1} q_n$. This completes the 
inductive construction.

Let $q= (u, T)$ be the fusion of the $q_n$'s 
(see Definition \ref{def fusion for Q}), so 
$T = \langle t_i : i \in \omega \rangle$ with $t_i = t_i^{i+1}$ for 
all $i \in \omega$. 
We show $q$ is $(M, \mathbb{Q})$-generic by showing 
that for all $n \in \omega$, 
the set $D_n \intersection M$ is predense below
$q$. Towards this end 
let $r = (v, R)$ be an arbitrary extension of 
$q$; as $D_n$ is dense there exists $ w \subseteq 
\mathrm{int}(R)$ such that 
$(v \union w, R \setminus \max(w)) \in D_n$. Then as
$(v \union w, R \setminus \max(w)) \leq (u, T_{n+1} \setminus
\max(w)) \leq q_{n+1}$ and $q_{n+1}$ is preprocessed for 
$D_n$ and $\max(w)$, already
$r' = (v \union w, T_{n+1} \setminus \max(w)) \in D_n$. 
Then $r' \in M$ and $r,r'$ are compatible, 
as witnessed by the condition $(v \union w, R \setminus \max(w))$. 

Next, we show $q \forces \card{\dot{Z}_n \intersection B} = 
\omega$ for every $n \in \omega$. Fix $n$, and let 
$(v, R)\leq q$ be arbitrary; it suffices to show that 
for every $k \in \omega$ there exists 
an extension of $(v,R)$ which forces 
$(\dot{Z}_n \intersection B) \setminus k \neq \emptyset$.
Find $i \in \omega$ such that 
$i > k$, $ v \subseteq i$, $\dot{Z}_n = \dot{Z}_i$ and 
$s = \mathrm{int}(R) \intersection 
\mathrm{int}(t_i)$ is $t_i$-positive.
The fact that $s$ is $t_i = t_i^{i+1}$-positive and 
$r \leq q \leq q_{i+1}$ implies, by item (3),
that there exists $w \subseteq s$ 
such that 
$$(v \union w, \langle t_j^{i+1} : j > \max(s) \rangle )
\forces m_{i+1} \in \dot{Z}_i,$$
where $m_{i+1} \in B$ and $m_{i+1} \geq i > k$. 
Then $(v \union w, R \setminus s)$ is an extension both of 
$(v,R)$ and of the condition
$(v \union w,
\langle t_j^{i+1} : j > \max(s) \rangle )$, so by the latter, 
forces $m_{i+1} \in (B \intersection \dot{Z}_n) \setminus k$.
This completes the proof that 
$q$ is an $(M, \mathbb{Q}, \calA, B)$-generic condition.
\end{proof}

The following is a straightforward density argument.
\begin{lemma} \label{Q adds unsplit}
Let $G$ be $\mathbb{Q}$-generic over $V$, let 
$\calS \subseteq \infsubsets$ be an element of $V$, 
and let $a = \set{s \subseteq \omega}[\exists T \ (s,T) \in 
G]$. Then for all $b \in \calS$, $(b$ does not split $a)^{V[G]}$. 
\end{lemma}

    \begin{theorem} [{\cite[Theorem 3.1]{Shelah84}}] \label{Theorem a < s}
    (\CH) Let  $\langle \mathbb{P}_\alpha, \dot{\mathbb{Q}}_\beta : 
   \alpha \leq \omega_2, \beta < \omega_2 \rangle$
   be a countable support iteration such that for all 
   $\alpha < \omega_2$, $\dot{\mathbb{Q}}_\alpha$ is a
   $\mathbb{P}_\alpha$-name for the partial order 
   $\mathbb{Q}$ of Definition \ref{def creature}. 
   Let $G$ be $\mathbb{P}_{\omega_2}$-generic over $V$. 
   Then 
   $V[G] \models \aleph_1 = \mathfrak{a} < 
      \mathfrak{s} = \aleph_2$.
    \end{theorem}

\begin{proof}
Since $V \models $\textsf{CH}
we can fix a tight mad family $\calA \in V$.
Let $G$ be $\mathbb{P}_{\omega_2}$-generic over $V$. For all $\alpha < \omega_2$, by Theorem \ref{Prop Q strongly preserves tightness}, $\mathbb{P}_\alpha$ forces that
 $\dot{\mathbb{Q}}_\alpha$ is a proper forcing which 
 strongly preserves the tightness of $\calA$, so using 
 Lemma \ref{csi iterations strongly preserve tightness}, 
 $\mathbb{P}_{\omega_2}$ is proper and
preserves the tightness and hence maximality of 
$\calA$.
Therefore $\calA$ witnesses 
$(\mathfrak{a} = \aleph_1)^{V[G]}$.
Let $\calS \subseteq \infsubsets$ be family of cardinality
$< \aleph_2$. 
Since $\mathbb{P}_{\omega_2}$ is $\aleph_2$-cc (see \cite[Theorem 2.10]{Abraham2010}), there exists $\alpha < \omega_2$ such that 
$\calS \in V[G_\alpha]$, where 
$G_\alpha = G \intersection \mathbb{P}_\alpha$ is 
$\mathbb{P}_\alpha$-generic over $V$. By definition of 
$\dot{\mathbb{Q}}_\alpha$ and Lemma \ref{Q adds unsplit},
in $V[G_{\alpha + 1}]$, $\calS$ is not a splitting family,
so also this holds in $V[G]$. 
Therefore $(\mathfrak{s} = \aleph_2)^{V[G]}$. 
\end{proof}

\section{Sacks coding and tightness}
Recall that under $V=L$ there exists a $\Delta_2^1$ wellorder of the reals \cite{Godel39}; this complexity is optimal by the
Lebesgue measurability of analytic sets. Conversely, Mansfield's theorem states that if there 
exists a $\Sigma_2^1$ wellordering of the reals, then all 
reals are constructible.
Using a finite support iteration of
ccc forcings, Harrington \cite[Theorem B]{Harrington77} 
showed that a $\mathbf{\Delta}_3^1$
wellordering is consistent with $\neg$\textsf{CH}
and Martin's Axiom (\textsf{MA}); Friedman and Caicedo 
showed that the Bounded Proper Forcing Axiom 
(\textsf{BFPA}) and
$\omega_1 = \omega_1^L$ imply the existence of a 
$\Sigma_3^1$ wellorder of the reals \cite{CF2011}.
However in these last two constructions, the forcing
axioms rendered all cardinal characteristics equal 
to $\mathfrak{c}$. The question of projective wellorderings of the reals
in models with nontrivial structure of cardinal characteristics
of the continuum was first addressed by Fischer and Friedman
in 2010 \cite{FF2010}. Using a countable support iteration of 
$S$-proper forcing notions they showed a $\Delta_3^1$ wellorder
of the reals is compatible with 
$\mathfrak{c} = \aleph_2$ and each of the following 
inequalities: $\mathfrak{d} < \mathfrak{c}$, 
$\mathfrak{b} < \mathfrak{a} = \mathfrak{s}$, 
$\mathfrak{b} < \mathfrak{g}$. This was made
possible by defining a new forcing notion, \emph{Sacks coding},
which uses Sacks reals to code the wellorder and gives 
a way of forcing a $\Delta_3^1$ wellorder with an 
$\omega^\omega$-bounding iteration. 

Definable mad families of size $\mathfrak{c}$ can be obtained from 
a definable wellorder of $\mathbb{R}$, bringing into focus the study if the 
projective complexity of such families.
Mathias \cite{MATHIAS197759} was the
first to do this, when in 1969 he showed no 
analytic almost disjoint family can be maximal. Subsequent 
work revealed further tension between 
mad families and determinacy assumptions;
see, for example,  \cite{NN18}, \cite{Bakke_Haga_2021},
\cite{Tornquistdefinability}, and references therein. 
On the other hand, in $L$ there are $\Sigma_2^1$ mad families, as a recursive construction along the 
$\Sigma_2^1$-good wellorder $\leq_L$ of the 
constructible reals yields such a family. This was significantly improved by 
Miller \cite[Theorem 8.23]{miller1989}, 
who obtained a coanalytic mad family under $V=L$. 

It is interesting to ask whether 
one can obtain a cardinal preserving generic extension with
a $\Delta_3^1$ wellorder of the reals and $\neg$\CH , while simultaneously 
controlling values of cardinal characteristics as well 
as the definability of the witnesses to those characteristics.
In 2022
Bergfalk, Fischer, and Switzer established that a 
$\Delta_3^1$ wellorder of the reals is consistent with
$\mathfrak{a} = \mathfrak{u} < \mathfrak{i}$, 
$\mathfrak{a} = \mathfrak{i} < \mathfrak{u}$,
and $\mathfrak{a} < \mathfrak{u} = \mathfrak{i}$, with
the added feature that the witnesses to the 
cardinal characteristics of value $\aleph_1$ can 
be taken to be coanalytic. In particular, they show that 
Sacks coding 
strongly preserves tightness of {\emph{nicely definable}} 
tight mad families (\cite[Lemma 4.3]{bergfalk2022projective}). 
Making use of this result and  Theorem
\ref{Prop Q strongly preserves tightness} we obtain:

\begin{theorem} \label{delta 13 with a < s}
 It is consistent that $\mathfrak{a} = \aleph_1 < 
 \mathfrak{s} = \aleph_2$ and that there exists a 
 $\Delta_3^1$ definable wellorder of the reals, and a 
 $\Pi_1^1$ tight mad family of size $\aleph_1$. 
 \end{theorem}
The above answers \cite[Questions 7.$(1)\&(2)$]{FF2010}.
Our strategy in proving the above is  to 
force with a countable support iteration over $V=L$ 
of $S$-proper forcings which strongly 
preserve the tightness of a coanalytic tight mad family 
$\calA \in V$. 
Along the way we construct the
wellorder $<_G = \bigcup_{\alpha < \omega_2} <_\alpha$ 
by defining the initial segments 
$<_\alpha$, where 
$<_\alpha$ wellorders the reals of $L^{\mathbb{P}_\alpha}$
using a canonical wellordering via $\leq_L$ of 
$\mathbb{P}_\alpha$-names for reals. Using appropriate
bookeeping, at stage $\alpha$ we add a Sacks-generic 
real coding a pair of reals $x, y \in V^{\mathbb{P}_\alpha}$
such that $x <_\alpha y$. The way in which the 
$\alpha$-th generic real codes these intial stages of 
the iteration  provides 
the $\Delta_3^1$-definability of the wellorder. More precisely
at stage $\alpha$, 
the Sacks-generic real $r_\alpha$ will code
a  sequence 
$\vec{C}_\alpha = \langle C_{\alpha+m} : m \in \Delta(x \ast y)\rangle $
of generic clubs in 
$\omega_1$, as well as a set $Y \subset \omega_1$,
where $\Delta(x \ast y) \subseteq \omega$ is a recursive
coding of the pair $(x,y)$, thus encoding a 
a pattern  of stationary/nonstationary in apriori fixed sequence 
$\langle S_{\alpha +m } : m \in \omega \rangle \in L$, 
of stationary costationary subsets of $\omega_1^L$.
Note that, if $r_\alpha$ codes the pair
$(x,y)$, then
$L[r_\alpha] \models 
\Delta(x \ast y) \subseteq \set{m \in \omega}[
S_{\alpha + m} \text{ is nonstationary} ]$.
If we can guarantee that for  
no $m \not\in \Delta(x \ast y)$, 
$S_{\alpha + m}$ loses its stationarity in 
the final extension, 
then the stationarity-nonstationarity pattern 
encodes the wellorder $<_G$ as follows:
\begin{align*}
x <_G y \Leftrightarrow
\exists \alpha < \omega_2 \ L[r_\alpha] \models  & \Delta(x \ast 
y) = \\ 
&\set{m \in \omega}[
S_{\alpha + m}{\text{ is nonstationary}}]
\end{align*}
To make the latter a (lightface) projective definition
of $<_G$, 
the set $Y$ is added after $\vec{C}_\alpha$ 
in order to localize the generic nonstationarity 
of each $S_{\alpha+m}$ to a large class of countable transitive 
\textsf{ZF}$^-$
models (Definition \ref{definition suitable model} below); this uses what is sometimes referred to as 
``David's Trick''. Roughly, this allows us to bound 
the initial existential quantification to
$\omega_2^\calM$, where $\calM$ belongs to the
said class of countable transitive models, and 
$\calM$ contains the real $r_\alpha$. 
We proceed by introducing some key to our construction forcing notions.

\subsection{Club shooting}
Baumgartner, Harrington, and Kleinberg introduced  in \cite{BHK76}
a cardinal preserving 
forcing notion which, given a stationary costationary 
$S \subseteq \omega_1$, adds a closed unbounded 
$C \subseteq \omega_1$ such that $C \intersection S = 
\emptyset$. The forcing is often referred to as
\emph{club shooting}.

\begin{definition} \label{def club shooting} 
Let $S \subseteq \omega_1$ be a stationary costationary set. Define 
$Q(S)$ to be the partial order consisting of closed, bounded 
subsets of $\omega_1 \setminus S$, ordered by end-extension. 
\end{definition}

\begin{lemma}
The following hold: 
\begin{enumerate}
\item $Q(S)$ is $\omega$-distributive, thus does not add new 
reals.
\item Let $G$ be $Q(S)$-generic. Then
 $C_G = \bigcup \set{d \in Q(S)}[d \in G]$ is a 
 club in $\omega_1^{V[G]}$ witnessing $(\check{S} $ is nonstationary$)^{V[G]}$. 
\end{enumerate}
\end{lemma}

\begin{proof}
See \cite[Chapter 25]{Jech} or \cite[Section 6]{Cummings2010}. 
\end{proof}

By item (2) above, $Q(S)$ is not a proper forcing notion, 
however, it still retains many of the desirable properties
of a proper forcing. 

\begin{definition} \label{Definition 
S-proper}
Let 
$S \subseteq \omega_1$ be a stationary set.
A forcing notion $\mathbb{P}$ is $S$-\emph{proper} 
if for all countable elementary submodels 
$M \prec H_\theta$, with 
$\theta$ sufficiently large, and such that 
$M \cap \omega_1 \in S$, for every $ p \in \mathbb{P} \intersection M$
there is $q \leq p$ which is $(M, \mathbb{P})$-generic.
\end{definition}

A proof of the following can be found in 
\cite[Theorem 3.7]{Goldsterntaste}
\begin{lemma} \hfill
\begin{enumerate}[(1)]
\item If $S \subseteq \omega_1$ is stationary and 
$\mathbb{P}$ is $S$-proper, then $\mathbb{P}$ preserves $\omega_1$ as well as 
the stationarity of any stationary subset of 
$S$. 
\item $Q(S)$ is $(\omega_1 \setminus S)$-proper.
\end{enumerate}
\end{lemma}  


\begin{lemma}  \label{preservation of S proper omega bounding} \hfill
\begin{enumerate}[(1)]
\item
Let $\langle  \mathbb{P}_\alpha, \dot{\mathbb{Q}}_\beta 
: \alpha \leq \delta, \beta < \delta \rangle$ 
be a countable support iteration such that for all $\alpha < \delta$, 
$\forces_{\mathbb{P}_\alpha}$``$ \dot{\mathbb{Q}}_\alpha$ is an $S$-proper poset''. 
Then $\mathbb{P}_\delta$ is $S$-proper. 
Moreover, if $\calA$ is a tight mad family
in the ground model and for all $\alpha < \omega_2$, 
$\forces_{\mathbb{P}_\alpha}$``$ \dot{\mathbb{Q}}_\alpha$ strongly preserves the tightness of 
$\check{\calA}$'', then also
$\mathbb{P}_\delta$ strongly preserves the tightness of 
$\calA$. 
\label{chain condition S proper}  \label{Lemma CH preserved by S proper}
\item  Assume \CH, and let $\langle \mathbb{P}_\alpha, \dot{\mathbb{Q}}_\beta 
 : \alpha \leq \delta, \beta < \delta \rangle$ 
 be a countable support
iteration of $S$-proper posets of length $\delta \leq \omega_2$, such that 
 for all $\alpha < \delta$, 
$\forces_{\mathbb{P}_\alpha}$``$ \card{\dot{\mathbb{Q}}_\alpha} 
= \omega_1$''.
Then $\mathbb{P}_\delta$ is $\aleph_2$-cc.  
If also $\delta < \omega_2$, then \textsf{CH} holds in 
$V^{\mathbb{P}_\delta}$.
\end{enumerate}
\end{lemma} 
\begin{proof}
Easy modification of corresponding properness arguments.
\end{proof}

\subsection{Localization}
The forcing notion known as Localization has its roots in 
Ren\'e David's work \cite{RDavid} on absolute $\Pi_2^1$
singletons.
Below we let
  \ZF$^-$ denote  \ZF\  without the Powerset Axiom.

\begin{definition}\label{definition suitable model}
A transitive model $\calM$ of \ZF$^-$ is called \emph{suitable} 
if $\omega_2^\calM$ exists and $\omega_2^\calM = \omega_2^{L^\calM}$.
\end{definition} 

Throughout the rest of this section we assume 
$V$ is a generic extension of $L$ via a cofinality preserving 
forcing extension. 

 \begin{definition} \label{localization forcing def}
 For $X \subseteq \omega_1$ and a $\Sigma_1$-sentence
 $\varphi(\omega_1, X)$ with parameters 
 $\omega_1$ and $X$ 
 such that $\varphi$ holds in all
 suitable models $\calM$ with $\omega_1, X \in \calM$,
 denote by $\mathcal{L}(\phi)$ the set of all functions 
 $r \from \card{r} \to 2$ where $\card{r} = 
 \mathrm{dom}(r)$ 
 is a countable limit ordinal, and such that: 
 \begin{enumerate}
 \item if $\gamma < \card{r}$ then $\gamma \in X$ if and only 
 if $r(2\gamma)= 1$; 
 \item if $\gamma \leq \card{r}$ and $\calM$ is a countable 
 suitable model such 
 that $\gamma = \omega_1^\calM$ and 
 $r \restriction \gamma \in \calM$, then 
 $\calM \models \varphi(\gamma, X \intersection \gamma)$. 
 \end{enumerate}

  The extension relation is end-extension. 
  
  \end{definition}
  
  Each $r \in \calL(\varphi)$ is an approximation 
  to the characteristic function of a subset $Y \subseteq 
  \omega_1$ such that 
  $\mathrm{Even}(Y) = \set{\gamma}[2\gamma \in Y] = X$. 
  The ``odd part'' of $r$, i.e. the values $r$ takes on 
  ordinals of the form $2\gamma +1$, is used for item 
  (1) below.
  
\begin{lemma} The following hold.
\begin{enumerate}
\item (\cite[Lemma 1]{FF2010}) \label{localization dense set}
For every $r \in \calL(\varphi)$ and countable limit ordinal 
$\gamma > \card{r}$, there exists $r' \leq r$ such that 
$\card{r'} = \gamma$. 

\item  (\cite[Lemma 2]{FF2010})\label{generic for localization} 
  If $G$ is $\mathcal{L}(\varphi)$-generic and $\calM$  is 
  a countable suitable model such that 
  $\bigcup G \restriction \omega_1^\calM \in \calM$,
  then $\calM \models \varphi(\omega_1^\calM, X \intersection 
  \omega_1^\calM)$. 
  
  \item (\cite[Lemma 3, Lemma 4]{FF2010})
  \label{Lemma localization is proper no new reals}
 $\mathcal{L}(\varphi)$
 is proper, and moreover 
does not add
 new reals. 
\end{enumerate}
\end{lemma}

\subsubsection{Sacks Coding}
Sacks coding, or coding with perfect trees, was first defined by Fischer and Friedman \cite{FF2010}. Recall a 
tree $T \subseteq 2^{<\omega}$ is 
\emph{perfect} if for all $s \in T$ there exists 
an extension $t$ of $s$ such that $t \in T$ 
and $t^\frown (i) \in T$ for each $i < 2$. 
\emph{Sacks forcing} is the partial order consisting of perfect
trees $T \subseteq 2^{< \omega}$, and 
a condition $S$ extends a condition $T$ if 
$S$ is a subtree of $T$. 
 Sacks forcing is often viewed as a very minimally destructive
 way of forcing $\neg$\CH ; the Sacks coding forcing 
 defined below 
 can be seen as a minimally 
 destructive way of forcing a $\Delta_3^1$ wellorder
 in the presence of large continuum. 
Throughout this section we assume $V=L[Y]$, where
$Y \subseteq \omega_1$ is generic over $L$ for a 
cardinal preserving forcing notion.

   \begin{definition}
    Fix $n^* \in \omega$. By induction on $i < \omega_1$, 
     define a sequence 
     $\overline{\mu} = \langle \mu_i : i \in \omega_1 \rangle$
     such that 
     $\mu_i$ is an ordinal and 
     $(\card{\mu_i} = \aleph_0)^L$, for each $i < \omega_1$. 
      Let $\mu_0 = 
     \emptyset$, and 
     supposing $\langle \mu_j : j < i \rangle$ has been 
     defined, let $\mu_i$ be the least ordinal $\mu > 
     \sup_{j < i} \mu_j $ such that : 
     \begin{enumerate}
     \item $L_\mu[Y \intersection i] \prec_{\Sigma_{n^*}^1}
     L_{\omega_1}[Y \intersection i]$; 
     \item $L_\mu[Y \intersection i] \models $ \ZFminus; 
     \item $L_\mu[Y \intersection i] \models$ ``$\omega$ is the
     largest cardinal''. 
     \end{enumerate}
     \end{definition}

     Let $\calB_i := L_{\mu_i}[Y \intersection i]$.
     For $r \in 2^\omega$, we say that $r$ \emph{codes} $Y$ 
    \emph{below} $i$ if for all $j < i$, 
    $j \in Y \Leftrightarrow \underbrace{L_{\mu_j}[Y \intersection j][ r]}_{=\calB_j[r]}
    \models \textsf{ZF}^-$. 
        For a tree $T \subseteq 2^{<\omega}$, let 
    $\card{T}$ denote the least $i < \omega_1$ such that
    $T \in \calB_i$. 
    
      \begin{definition}{(\cite[Definition 2]{FF2010})} 
      \label{def sacks coding}
 \emph{Sacks coding} is the partial order 
  $C(Y)$ consisting of perfect trees $T \subseteq 2^{<\omega}$
  such that $r$ codes $Y$ below 
  $\card{T}$ whenever $r$ is a branch through $T$.
   For $T_0, T_1 \in C(Y)$, 
   let $T_1 \leq T_0$ if and only if
  $T_1$ is a subtree of $ T_0$. 
  \end{definition}

\begin{remark}
The hidden parameter $n^*$ above is needed for the preservation results 
pertaining to definable combinatorial objects, e.g. 
item (5) of Fact \ref{properties of C(Y)} below. 
Specifically, $n^*$ is chosen to an upper bound on 
the complexity of the formula expressing all relevant combinatorial
properties which we want to reflect down to the models 
$\calB_i$. In
the present case of preserving a $\Pi_1^1$ tight mad family, 
$n^* = 5$ is sufficient. See also \cite[Remark 1]{bergfalk2022projective}.

\end{remark}


  \begin{remark}
  Let $G$ be $C(Y)$-generic over $L[Y]$. Since the definition of 
  $C(Y)$ is absolute, it still holds in $L[Y][G]$ that 
  if $T$ is any condition in $C(Y)$ and $r$ is a branch 
  through $T$, in $L[Y][G]$,
  $$
  Y \intersection \card{T} = 
  \set{ j < \card{T}}[\calB_j[r{]} \models \mathsf{ZF}^-]$$
  
  In particular the above holds for all 
  $j <  \sup \set{ \card{T}}[T \in G]$ whenever $r \in \bigcap G$.
  \end{remark}
  
  \begin{fact} \label{properties of C(Y)} The following hold. 
  \begin{enumerate}
  \item (\cite[Lemma 4]{FF2010}) For any $j \in \omega_1$ and 
  any $T \in C(Y)$ with $\card{T} \leq j$, there exists 
  $T' \leq T$ such that $\card{T'} = j$. 
  \item (\cite[Lemma 7]{FF2010}) $C(Y)$ is proper.
  \item (\cite[Lemma 6]{FF2010}) Let $G$ be $C(Y)$-generic
  over $V$ and let $r := \bigcap G$. Then in $V[G]$, 
  $r$ codes the set $Y$ in the sense that for all $j < \omega_1$, 
  $j \in Y \Leftrightarrow \calB_j[r] \models \mathsf{ZF}^-.$ 
  \item (\cite[Lemma 8]{FF2010}) $C(Y)$ is 
  $\omega^\omega$-bounding.
  \item (\cite[Lemma 4.3]{bergfalk2022projective}) Countable 
  support iterations of $C(Y)$ 
  preserve the tightness of $\Pi_1^1$ tight mad families 
  which provably consist of constructible reals.
  \end{enumerate}
  
  \end{fact}

 \subsection{A $\Delta_3^1$ long wellorder}
 With the ingredients provided by the previous sections we
 proceed with the proof of
  Theorem \ref{delta 13 with a < s}. 
  Recall that $\lozenge$ is the assertion that there exists 
  a sequence $\vec{A} = \langle A_\alpha : \xi < \omega_1 \rangle $
  where
  $A_\xi \subseteq \xi$
  for each $\xi < \omega_1$ and for any $X \subseteq \omega_1$, 
  the set $\set{\xi < \omega_1}[X \intersection \xi = A_\xi]$ is stationary. $\vec{A}$ is called a $\lozenge$-sequence, and 
  such a sequence can be constructed in $L$ so that
  the sequence is $\Sigma_1$-definable over $L_{\omega_1}$; 
  see, for example, \cite[Theorem 3.3]{Devlin}.

  \begin{proposition}[{\cite[Lemma 14]{FF2010}}] \label{statcostat sequence}
  Assume $\lozenge$ holds. Then there exists a sequence
  $\vec{S} = \langle S_\alpha : \alpha < \omega_2 \rangle$
  which is $\Sigma_1$-definable over $L_{\omega_2}$
  consisting
  of sets $S_\alpha \subseteq \omega_1$ 
  that are stationary costationary in $L$, and 
  are almost disjoint in the sense that 
  $\card{S_\alpha \intersection S_\beta} < \omega_1$ for 
  all distinct $\alpha, \beta < \omega_2$.
  Furthermore there exists a stationary 
  $S_{-1} \subseteq \omega_1$ such that
  $S_{-1} \intersection S = \emptyset$ for all
  $S \in \vec{S}$. 
  Moreover, if $\calM$, $\calN$ are suitable models such that 
  $\omega_1^\calM = \omega_1^\calN$, then
   $\vec{S}^\calM$ and $\vec{S}^\calN$ coincide 
  on $\omega_2^\calM \intersection \omega_2^\calN$. If  
  $\calM$ is suitable with $\omega_1^\calM = \omega_1$, 
  then $\vec{S}^\calM = \vec{S} \restriction \omega_2^\calM$. 
  
  \end{proposition}

We assume $V=L$, and therefore fix $\vec{S}$ and 
$S_{-1}$ as above.

\begin{lemma}[{\cite[Lemma 14]{FF2010}}]
There exists $F \from \omega_2 \to L_{\omega_2}$ such that for all $a \in L_{\omega_2}$, $F^{-1}(a)$ is unbounded
in $\omega_2$, and $F$ is $\Sigma_1$-definable over 
$L_{\omega_2}$. 
Moreover, if $\calM$, $\calN$ are suitable models such that 
  $\omega_1^\calM = \omega_1^\calN$, then
   $F^\calM$ and $F^\calN$ coincide 
  on $\omega_2^\calM \intersection \omega_2^\calN$. If  
  $\calM$ is suitable with $\omega_1^\calM = \omega_1$, 
  then $F^\calM = F \restriction \omega_2^\calM$. 
\end{lemma}

Fix $F$ as above. By Fact \ref{properties of C(Y)}, item (5), tight mad families are 
indestructible by countable support iterations of $C(Y)$ when
they satisfy an extra definability assumption; such a mad
family can be constructed in $L$ by the following.
 
\begin{lemma}[{\cite[Lemma 4.2]{bergfalk2022projective}}]
\label{Lemma exists coanalytic tight}
If $V=L$ then there is a $\Pi_1^1$ tight mad family $\calA$
 such that \ZFC\
proves $\calA$ is a subset of $L$.
\end{lemma}

By recursion on $\alpha < \omega_2$, define 
a countable support iteration 
 $\langle \mathbb{P}_\alpha, \dot{\mathbb{Q}}_\beta : \alpha \leq \omega_2, \beta < \omega_2 \rangle$, where
$\mathbb{P}_0$ is taken to be the trivial poset. Suppose
$\mathbb{P}_\alpha$ has been defined, and let 
$G_\alpha$ be $\mathbb{P}_\alpha$-generic over $L$.
The wellorder $<_\alpha$ on $L[G_\alpha]$
has a natural definition using the global wellorder $<_L$
of the universe $L$ and the collection of $\mathbb{P}_\alpha$-names for reals. We can assume any name for a real $\dot{x}$ is
\emph{nice}, in that it is determined by countably many 
maximal antichains $A_n \subseteq \mathbb{P}_\alpha$ consisting 
of conditions deciding $\dot{x}(n)$. This allows us to suppose that if $\alpha < \beta < \omega_2$ and 
  $\dot{x}$ is a $\mathbb{P}_\beta$-name which is not
  a $\mathbb{P}_\alpha$-name, then all $\mathbb{P}_\alpha$-names
  precede $\dot{x}$ with respect to $<_L$, 
  as it takes longer to construct $x$. 
  Whenever $x \in L[G_\alpha]$ is a real, let 
  $\gamma_x$ denote the minimal $\gamma$ such that
 $x$ has a nice $\mathbb{P}_\gamma$-name, and define $\sigma_x^\alpha$ to be 
  the $L$-minimal nice $\mathbb{P}_{\gamma_x}$-name for $x$. In this
  way we can understand the reals of $L[G_\alpha]$ by 
  considering the set 
  $N = \set{ \sigma_x^\alpha}[x \in L[G_\alpha{]} \intersection 
  \omega^\omega],$
  which is canonically wellordered by $<_L$;
  therefore for reals $x,y \in L[G_\alpha]$, 
  let 
  $x <_\alpha y$ if and only if $\sigma_x^\alpha <_L
  \sigma_y^\alpha$.
  Since
   $\sigma_x^\alpha = \sigma_x^\beta = 
  \sigma_x^{\gamma_x}$ for any $\beta < \alpha$, 
  $<_\beta$ is an initial segment of $<_\alpha$.
  Let $\dot{<}_\alpha$ denote a $\mathbb{P}_\alpha$-name 
  for $<_\alpha$. 
  
 Fix a recursive coding $\cdot \ast \cdot \from 
 \omega^\omega \times \omega^\omega \to \omega^\omega$ by 
 letting
 $x \ast y = 
  \set{2n}[n \in x] \union \set{2n+1}[n \in y],$
  and for a real $x$ define 
 $\Delta(x) := x \ast (\omega \setminus x)$.
 Lastly we fix an absolute way of coding the ordinals 
 of $\omega_2^L$. 
 \begin{definition} \label{sufficiently absolute code}
Let $\beta < \omega_2^L$ and $X \subseteq \omega_1^L$. 
We say $X$ is a \emph{sufficiently absolute code for }$\beta$ if
there exists a formula $\psi(x,y)$ such that for any suitable
model $\calM$ containing $X \intersection \omega_1^\calM$, 
there exists a unique 
$\overline{\beta} \in \omega_2^\calM$ such that 
$\psi(\overline{\beta}, X \intersection \omega_1^\calM)$ 
holds in $\calM$, 
and $\overline{\beta}= \beta$ in the case 
$\omega_1^\calM = \omega_1^{L}$. 
Moreover, if $\calM, \calN$ are any suitable models such that 
$\omega_1^\calM = \omega_1^\calN$ and $X \intersection 
\omega_1^\calM \in \calM \intersection \calN$, 
then it is the same $\overline{\beta} \in \omega_2^\calM 
\intersection \omega_2^\calN$ such that
$(\psi(\overline{\beta}, X \intersection \omega_1^\calM))^\calM$
and $(\psi(\overline{\beta}, X \intersection \omega_1^\calN))^
\calN$. 
\end{definition}

\begin{fact}{(\cite[Fact 5]{Friedmanzdomskyy})} \label{sufficiently absolute codes exist}
Sufficiently absolute codes exist for every 
$\beta < \omega_2^L$. 
\end{fact}

Henceforth we fix the formula $\psi(x,y)$ above.
 Working in  $L[G_\alpha]$, let
 $\dot{\mathbb{Q}}_\alpha = \dot{\mathbb{Q}}_\alpha^0 \ast 
 \dot{\mathbb{Q}}_\alpha^1$ be a 
 $\mathbb{P}_\alpha$-name for a two-step iteration in which
 $\dot{\mathbb{Q}}_\alpha^0$ is a $\mathbb{P}_\alpha$-name for
$\mathbb{Q}$ (Definition 
 \ref{def creature}), and 
 $\dot{\mathbb{Q}}_\alpha^1$ is a $\mathbb{P}_\alpha \ast \dot{\mathbb{Q}}_\alpha^0$-name 
 for the trivial forcing, unless the following occurs: 
 $F(\alpha)= \set{\sigma_x^\alpha, \sigma_y^\alpha}$ for some reals
 $x,y \in L[G_\alpha]$ such that $x
  <_\alpha y$. In this case set $x_\alpha = x$, $y_\alpha = y$ 
   and 
  define $\dot{\mathbb{Q}}_\alpha^1$ to be a
  $\mathbb{P}_\alpha \ast \dot{\mathbb{Q}}_\alpha^0$-name 
  for a three-step
  iteration $\dot{\mathbb{K}}_\alpha^0 \ast \dot{\mathbb{K}}_\alpha^1 
  \ast \dot{\mathbb{K}}_\alpha^2$, where:

\medskip
$(1)$ $\dot{\mathbb{K}}_\alpha^0$ is a $\mathbb{P}_\alpha
\ast \dot{\mathbb{Q}}_\alpha^0$-name for 
  the countable support iteration
  $\langle \mathbb{P}_{\alpha, \beta}^0, \mathbb{K}_{\alpha, n}^0 
  : \beta \leq \omega, n \in \omega
  \rangle$, where $\mathbb{K}_{\alpha, m}^0$ is a 
  $\mathbb{P}_{\alpha, m}^0$-name for $Q(S_{\alpha+m})$
  for all $m \in \Delta(x_\alpha \ast y_\alpha)$. 

\medskip
$(2)$ Let $R_\alpha$ be $\mathbb{Q}_\alpha^0$-generic
 over $L[G_\alpha]$, and let $H_\alpha$ be
 $\mathbb{K}_\alpha^0$-generic over 
 $L[G_\alpha \ast R_\alpha]$.  
 In $L[G_\alpha \ast R_\alpha \ast H_\alpha]$ fix:
 \begin{itemize}
 \item Subsets 
 $W_\alpha, W_\eta \subseteq \omega_1$ such that $W_\alpha$ 
 is a sufficiently absolute code for the ordinal 
 $\alpha$ and $W_\eta$ is a sufficiently absolute code 
 for an ordinal $\eta$ such that $L_\eta \models 
 \card{\alpha} \leq \omega_1$;
 \item A real $x_\alpha \oplus y_\alpha$ recursively coding the pair 
 $(x_\alpha, y_\alpha)$; 
 \item A subset $Z_\alpha \subseteq \omega_1$ coding $G_\alpha 
 \ast R_\alpha \ast H_\alpha$. 
\end{itemize}  
Fix a computable bijection 
	$\langle \cdot, \cdot, \cdot, \cdot \rangle \from 
	\mathcal{P}(\omega_1) \to \mathcal{P}(\omega_1)$, 
	and for $X \subseteq \omega_1$ and $i< 4$ write 
	$(X)_i$ for those elements of $\mathcal{P}(\omega_1)$ 
	such that $X = \langle (X)_i : i < 4 \rangle$.
	Let $X_\alpha = 
	\langle W_\alpha, x_\alpha \oplus y_\alpha, W_\eta, 
   Z_\alpha \rangle$ and let 
 $\varphi_\alpha = \varphi_\alpha(\omega_1, X_\alpha)$   be a 
 sentence with parameters $\omega_1$ and $X_\alpha$ such that 
 $\varphi_\alpha$ holds if and only if:
There exists an ordinal 
$\overline{\alpha} \in \omega_2$ such that $(X_\alpha)_0 = 
\overline{\alpha}$ and there exists 
 a pair $(x,y)$ such that $(X_\alpha)_2 = 
 x \oplus y$ and
  for all $m \in \Delta(x \ast y)$, 
 $S_{\overline{\alpha} + m}$ is nonstationary. 
Next, define 
  $\dot{\mathbb{K}}_\alpha^1$ to be a $(\mathbb{P}_\alpha \ast
  \dot{\mathbb{Q}}_\alpha^0 \ast
 \mathbb{K}_\alpha^0)$-name for 
 $\mathcal{L}(\varphi_\alpha)$.  

\medskip
$(3)$ Let $Y_\alpha$ be $\mathbb{K}_\alpha^1$-generic over 
 $L[G_\alpha \ast R_\alpha \ast H_\alpha]$. Then as 
 $\set{n \in \omega}[2n \in Y_\alpha]= X_\alpha$ and
 $X_\alpha$ codes $G_\alpha \ast R_\alpha \ast H_\alpha$, we have
 $L[G_\alpha \ast R_\alpha \ast H_\alpha \ast Y_\alpha] 
 = L[Y_\alpha]$. In this model, let 
 $\mathbb{K}_\alpha^2$ be a 
 $(\mathbb{P}_\alpha \ast \dot{\mathbb{Q}}_\alpha^0 \ast 
 \dot{\mathbb{K}}_\alpha^0 \ast 
 \dot{\mathbb{K}}_\alpha^1)$-name for 
 $C(Y_\alpha)$. 

This completes the definition of $\mathbb{P} = \mathbb{P}_{\omega_2}$. Then $\mathbb{P}$ is a countable 
support iteration such that each $\mathbb{P}_\alpha$ 
forces that  $\mathbb{Q}_\alpha$ is an $S_{-1}$-proper
forcing notion of size $\leq \omega_1$,
and by Theorem \ref{Prop Q strongly preserves tightness}
and Fact \ref{properties of C(Y)}, 
$\mathbb{Q}_\alpha$ is forced to 
strongly preserves the tightness of $\calA$. 
Therefore:
\begin{lemma} \label{Lemma full iteration adding delta13}
$\mathbb{P}$ is $S_{-1}$-proper, strongly preserves the 
tightness of $\calA$, and 
has the $\aleph_2$-cc. 
\end{lemma}

\begin{proof}
By Lemma \ref{preservation of S proper omega bounding}, Lemma \ref{csi iterations strongly preserve tightness}, and Lemma \ref{chain condition S proper}.
\end{proof}

Let $G$ be $\mathbb{P}$-generic over $L$ and $<_G = \bigcup <_\alpha$, where 
$<_\alpha = \dot{<}_\alpha^G$.
\begin{lemma} [{\cite{FF2010}}]\label{definability of wellorder}
Let $G$ be $\mathbb{P}$-generic over $L$, and let $x,y$ be reals in 
$L[G]$. Then $x <_G y$ if and only if:
$(*)$ there exists  a real 
$r$ such that for every countable suitable $\calM$ containing 
$r$ as an element, there exists 
$\overline{\alpha} < \omega_2^\calM$ such that for all 
$m \in \Delta(x \ast y)$, $S_{\overline{\alpha}+m}^\calM$ is 
nonstationary in $\calM$.
Hence, $<_G$ is a $\Delta_3^1$ definable wellorder of 
the reals in $L[G]$.
\end{lemma}

The complexity of $<_G$ is explicitly calculated as follows. 
We have
$x<_G y \Leftrightarrow \Phi(x,y),$
where $\Phi(x,y)$ is the formula

\begin{align*}
\underbrace{\exists r \in \mathcal{P}(\omega^\omega)}_\textrm{$\exists$} [ &
\underbrace{\forall \calM \text{ countable, suitable, }r \in \calM}_\textrm{$\forall$}\\
&\underbrace{(\exists \alpha < \omega_2^\calM}_\textrm{$\exists$} \ ( \underbrace{\calM \models  \forall m \in 
\Delta(x \ast y)  \ S_{\alpha+m} \in 
\mathrm{NS}_{\omega_1})}_\textrm{$\Delta_1^1$})]
\end{align*}

Thus, $\Phi(x,y)$ is a $\Sigma_3^1$ formula. 
However, 
$<_G$ is also a total wellorder, since if $x,y$ are any reals 
in $L[G]$, there exists  $\alpha = \max (\gamma_x, \gamma_y)<\omega_2$ such that $x = ({\sigma_x^\alpha})^G$ and  $y = ({\sigma_y^\alpha})^G$. 
Either $\sigma_x^\alpha <_L \sigma_y^\alpha$ or 
$\sigma_y^\alpha <_L 
\sigma_x^\alpha$; in the case $\neg(x<_G y)$ then we must have
$y <_G x$. Therefore the complement of 
$<_G$ is $\Sigma_3^1$-definable, giving that
 $<_G$ is $\Delta_3^1$-definable.

 
\medskip
\noindent
\emph{Proof of Theorem \ref{delta 13 with a < s}.}
Suppose $V=L$, and fix a $\Pi_1^1$ definable 
 tight mad family $\calA$ from 
 Lemma \ref{Lemma exists coanalytic tight}. Let 
 $\mathbb{P}$ be the $\omega_2$-length countable support iteration
 constructed in this section, and let $G$ be 
 $\mathbb{P}$-generic over $V$. By Lemma
 \ref{Lemma full iteration adding delta13}, $\mathbb{P}$
 is $S_{-1}$-proper.
 Define $<_G = \bigcup_{\alpha < \omega_2} \dot{<}_\alpha^G$; 
 by Lemma 
 \ref{definability of wellorder}, 
 $<_G $ is a 
 $\Delta_3^1$ wellorder of the reals. Since any set 
 $\calS \subseteq \infsubsets \intersection V[G]$ such 
 that 
 $\card{\calS} < \omega_2$ appears at some inital stage 
 $\alpha < \omega_2$ of the iteration, by 
 definition 
 $\dot{\mathbb{Q}}_\alpha^0$ is a 
 $\mathbb{P}_\alpha$-name for the forcing 
 $\mathbb{Q}$ of Definition \ref{def creature}, so
 $\calS$ is not splitting in 
 $V[G_{\alpha+1}]$. 
 Therefore $(\mathfrak{s} = 
 \aleph_2)^{V[G]}$. 
 Next notice that by Shoenfield absoluteness, 
 $\calA$ is defined by the same $\Pi_1^1$ formula in 
 $V[G]$ as in $V$. Moreover $\calA$ remains a tight
 mad family in $V[G]$ by Lemma \ref{Lemma full iteration adding delta13}, and therefore $\calA$ provides a witness for  $(\mathfrak{a} = \aleph_1)^{V[G]}$. This completes
 the proof. 
 \qed

    \section{Definable Spectra} 
    \label{FZ section}
      
  In this section we consider projective 
    mad families, 
    with a definition of optimal complexity, 
    which are of size
    $\kappa > \aleph_1$.
We briefly give an account of the work in this 
direction thus far.

Friedman and Zdomskyy \cite{Friedmanzdomskyy} established that a tight 
mad family with optimal projective definition is consistent with
$\mathfrak{b} = \mathfrak{c} = \aleph_2$ and this was 
extended by Fischer, Friedman and Zdomskyy \cite{FFZ11}
 to $\mathfrak{b} = 
\mathfrak{c} = \aleph_3$. By \cite{Raghavan}
no tight mad family can contain a perfect subset, and
the Mansfield-Solovay theorem states that any $\Sigma_2^1$ 
set is either a subset of $L$, or 
contains a perfect set of nonconstructible reals. Therefore,
no $\Sigma_2^1$ tight mad family can exist
in a model of $\mathfrak{b} > \aleph_1$; the optimal definition 
of a tight mad family of size greater than 
$\aleph_2$ 
is thus $\Pi_2^1$. 
Dropping the tightness requirement, Brendle and Khomskii 
\cite{BK12} showed $\mathfrak{b} = \mathfrak{c} \geq \aleph_2$, 
is consistent with a $\Pi_1^1$ mad family; one can
moreover have a $\Delta_3^1$ wellorder of the reals in such 
a model by
Fischer, Friedman, and Khomskii \cite{FFK13}. 
The first work on projective witnesses of size $\kappa$
when $\aleph_1 < \kappa < \mathfrak{c}$ is done by Fischer, 
Friedman, Schrittesser, and T\"ornquist \cite{FFST}; 
again by the Mansfield-Solovay theorem, the best
possible complexity of such an object is $\Pi_2^1$. 

So far the attention has been on finding definable 
mad families witnessing
the value of $\mathfrak{a}$,
the minimal element of the set
$\mathrm{spec}(\mathfrak{a}) = \set{ \card{\calA} }[
\calA \subseteq \infsubsets \text{ is mad}].$
The study of this set was pioneered by Hechler \cite{Hechler72}
and followed up by Blass \cite{Blass94}, and 
Shelah and Spinas \cite{Shelahspinas}. Similar considerations 
have since been taken for the cardinals $\mathfrak{a}_T$
and $\mathfrak{i}$; see \cite{FS25}, \cite{Brian24} for the former, and 
\cite{Fischershelah19}, \cite{Fischershelah22} for the latter.
Given that now there is fairly substantial knowledge
of how to control which $\kappa \in \mathrm{spec}(\mathfrak{a})$
as well as the definability of mad families of size 
$\mathfrak{a} = \min (\mathrm{spec}(\mathfrak{a}))$
it is natural we try to combine these lines of inquiry, 
by asking that for every $\kappa \in \mathrm{spec}(\mathfrak{a})$
there is a projective mad family of size
$\kappa$ with an optimal definition.  
Here we obtain:

\begin{theorem}\label{Theorem FZ section 2 sizes}
    It is consistent with $\mathfrak{a} = \aleph_1 < \mathfrak{c}
 = \aleph_2$ that there exists a $\Pi_1^1$-definable tight mad
 family of size $\aleph_1$, and a $\Pi_2^1$-definable tight mad
 family of size $\aleph_2$.
\end{theorem}


The strategy in proving the above statement will be as follows. We begin in a model of 
$V=L$ and fix a $\Pi_1^1$ tight mad family 
$\calA_1$. We recursively define a countable
support iteration $\langle \mathbb{P}_\alpha, \dot{\mathbb{Q}}_\beta : \alpha \leq \omega_2, \beta < \omega_2 \rangle $,
along the way constructing a $\Pi_2^1$ tight mad family 
$\calA_2$ consisting 
of $\omega_2$-many $\mathbb{Q}_\alpha$-generic reals $a_\alpha$,
arising from a certain forcing notion 
$\mathbb{K}_\alpha$, and such that 
each $\mathbb{Q}_\alpha$ is forced by $\mathbb{P}_\alpha$ 
to be an $S$-proper poset strongly preserving
the tightness of $\calA_1$. 
The tightness of $\calA_2$ is taken care of by appropriate
 bookeeping of those countable sequences of reals appearing
 in $\mathcal{I}(\calA^2_\alpha)^+$, where $\calA^2_\alpha$
 denotes the part of $\calA_2$ constructed up to 
 stage $\alpha$.
The definition will be a slight modification of the iteration 
defined in the proof of 
 \cite[Theorem 1]{Friedmanzdomskyy}. 
 The main modification is that, whereas \cite{Friedmanzdomskyy} added Hechler reals cofinally often along the
 iteration, so to increase $\mathfrak{b}$ and ensure
 there are no mad families of size $\aleph_1$, 
 in the present construction, cofinally often we take 
 $\mathbb{Q}_\alpha$ to be some $S$-proper forcing of 
 size $\aleph_1$. $\mathbb{Q}_\alpha$ is reserved for future
 applications.
The intention behind this use of Hechler forcing in \cite{Friedmanzdomskyy} 
was pointed out in \cite[Question 18]{Friedmanzdomskyy},
suggesting that it was unknown to the authors 
whether the iterand $\mathbb{K}_\alpha$ itself
added a dominating real. Note that Proposition  \ref{main} below will show that 
this is not the case.




The $\Pi_2^1$ definability of 
$\calA_2$ is achieved similarly to that of the 
$\Delta_3^1$ wellorder of the previous section, 
however differs in subtle but important ways.
Namely, using $\vec{S} = \langle S_{\alpha} : \alpha < \omega_2 \rangle$ and $S_{-1}$, as in Proposition \ref{statcostat sequence},
the $\mathbb{K}_\alpha$-generic real $a_\alpha$ will code a stationary/nonstationary
pattern into the $\alpha$th $\omega$-block of $\vec{S}$
by coding a sequence of clubs
$\vec{C}_\alpha = \langle C_{\alpha+m} : m \in \Delta(a_\alpha) \rangle$, which are also added by $\mathbb{K}_\alpha$. 
Therefore conditions will consist of a finite part, making 
approximations to $a_\alpha$, and an infinite part, 
making countably many approximations to a club in 
$\omega_1 \setminus S_{\alpha+m}$, for $m$ belonging 
to the finite part. This simultaneous construction must be carried out carefully to ensure $S_{\alpha+m}$ remains stationary
when $m \not\in \Delta(a_\alpha)$. The lightface
projective definition is obtained again relying on the localization 
techniques appearing in the last section, so 
$\mathbb{K}_\alpha$ must also add $\langle Y_{\alpha+m} : m \in \Delta(a_\alpha) \rangle$ of sufficiently absolute 
codes for $\alpha$. 
In the final extension $V[G]$,
for all $a \in V[G] \intersection \infsubsets$ we have
\begin{align*}
a \in \calA_2 \Leftrightarrow & \forall \calM (\calM \text{ is 
a countable suitable model and } a\in \calM) \\
&\  \exists \overline{\alpha}< \omega_2^\calM \ 
 ( \calM \models \forall m \in \Delta(a) \ S_{\overline{\alpha} + m} 
\text{ is nonstationary}).
\end{align*}
The right hand side is in the form 
$\forall \exists$ and thus is a $\Pi_2^1$ formula.

For the proof of Theorem \ref{Theorem FZ section 2 sizes} we
first establish preliminaries. Throughout the rest
of the section we work in a model of $V=L$ unless explicitly
stated otherwise.
Fix a coanalytic tight mad family $\calA_1$, as well as 
$\vec{S} = \langle S_\alpha: \alpha < \omega_2 \rangle$ and  $S_{-1} \subseteq 
  \omega_1$ given by Proposition 
  \ref{statcostat sequence}. 
  Let $F \from \mathrm{Lim}(\omega_2) \to L_{\omega_2}$ be such 
  that $F^{-1}(x)$ is unbounded in $\omega_2$ for all 
  $x \in L_{\omega_2}$. By \cite[Lemma 14]{FF2010}, we can take 
  $\vec{S}$ and $F$ to be $\Sigma_1$-definable over 
  $L_{\omega_2}$, 
  and moreover that whenever $\calM, \mathcal{N}$ are suitable 
  models with $\omega_1^\calM = \omega_1^{\mathcal{N}}$,
  then $\vec{S}^\calM$ and $\vec{S}^{\mathcal{N}}$ agree on 
  $\omega_2^\calM \intersection \omega_2^{\mathcal{N}}$.
Finally, for all $\alpha < \omega_2^L$, let $X_\alpha$ denote a 
sufficiently absolute code for 
$\alpha$.
Whereas the coding of the $\Delta_3^1$ wellorder was
achieved by the $C(Y)$-generic reals, the generic reals in the 
current context will result from almost disjoint coding, a
method developed by Solovay and Jensen \cite{solovayjensen}, and 
for which we give a general definition.

\begin{definition} \label{def ad coding}
Let $\calR$ be an almost disjoint family in $V$, and 
let $X\in V$ be a subset of $\omega_1$. The \emph{almost disjoint coding of }$X$\emph{with respect to }
$\calR$ is the partial order $\mathbb{P}_{\calR}(X)$
consisting of conditions $(s,F)$ such that 
$s$ is a finite subset of $\omega$ and $F$ is a finite subset 
of $\set{r_\xi}[\xi \in X] \subseteq \calR$. The extension relation is defined by letting 
$(t,G) \leq (s,F)$ if and only if 
\begin{enumerate}
\item $t$ end-extends $s$, $G \supseteq F$; 
\item For all $\xi \in X$ and $r_\xi \in F$, $(t \setminus s )
\intersection r_\xi = \emptyset$.
\end{enumerate}
\end{definition}

\begin{fact}
$\mathbb{P}_\calR(X)$ is $\sigma$-centered. If 
$G$ is $\mathbb{P}_\calR(X)$-generic over $V$, let
$a:= \bigcup \set{ s}[ \exists F (s,F) \in G].$ Then 
$G$ and $a$ are mutually definable in the generic extension, 
and in $V[G] = V[a]$, for all $\xi < \omega_1$, 
$  \xi \in X$ if and only if $ a \intersection r_{\xi}$ is 
finite.
\end{fact}


For proofs of the above, see, for example, \cite{Harrington77} or \cite[Example IV]{Jech2nd}. 
Using \cite[Proposition 3]{Friedmanzdomskyy} we fix an 
almost disjoint family
$\mathcal{R} = \set{R_{\langle \eta, \xi \rangle}}[\eta \in \omega
\cdot 2, \xi \in \omega_1]$ which is $\Sigma_1$-definable 
over $L_{\omega_1}$, and 
such that for every suitable model $\calM$, 
$\calR \intersection \calM = \set{R_{\eta, \xi}}[ \eta \in 
\omega \cdot 2, \xi \in \omega_1^\calM]$. 
For
technical reasons in the coding, we
modify the definition of the function $\Delta \from \omega^\omega \to 
\omega^\omega$ given in the previous section by letting, for $s \subseteq \omega$, finite or infinite, 
$\Delta(s) = \set{2n+1}[n \in s] \union \set{2n+2}[n \in (\sup s 
\setminus s)].$ Let $C(s) = \Delta(s) \union (\omega \setminus \max(\Delta(s))).$ We think of 
$C(s)$ as the ``coding area" associated with $s$. Denote by
$E(s), O(s)$ the sets 
$\set{s(2n)}[2n < \card{s}]$ and $\set{s(2n+1)}[2n+1 < \card{s}]$
respectively.
For a limit ordinal $\gamma$ and a function 
$r \from \gamma \to 2$, let $\mathrm{Even}(r) = \set{\alpha < \gamma}
[r(2\alpha) = 1]$ and $\mathrm{Odd}(r) = \set{\alpha < \gamma}
[r(2\alpha + 1) = 1]$. 
Lastly, for ordinals $\alpha < \beta$, let $\beta 
-\alpha$ denote the ordinal $\gamma$ 
such that $\alpha + \gamma = \beta$. If 
$B$ is a set of ordinals, let $B - \alpha = 
\set{ \beta - \alpha }[\beta \in \beta]$. 
If $\delta$ is an indecomposable ordinal then
$(\alpha + B) \intersection \delta - \alpha = 
B \intersection \delta$.

Continuing with the definition of $\langle \mathbb{P}_\alpha, \dot{\mathbb{Q}}_\beta
: \alpha \leq \omega_2, \beta < \omega_2 \rangle$ and the 
construction of a $\Pi_2^1$ tight mad family 
$\calA_2 = \set{a_\alpha}[\alpha < \omega_2
]$, 
suppose $\alpha < \omega_2$ and $\mathbb{P}_\alpha$ has been defined.
Let 
$G_\alpha$ be a $\mathbb{P}_\alpha$-generic filter, and let 
$\dot{\calA^2_\alpha}$ be a $\mathbb{P}_\alpha$-name for the set of elements 
of $\calA_2$ constructed up to stage $\alpha$. 
The density arguments of Lemma \ref{lem7} require the 
following inductive assumption:
\begin{equation} \label{eq:coding} \tag{$\ast$}
    \forall r \in \mathcal{R} \forall A' \in [\calA_\alpha^2]^{< \omega}
\ (\card{\mathrm{E}(r) \setminus \bigcup A' } = 
\card{\mathrm{O}(r) \setminus \bigcup A'} = 
\aleph_0 ),
\end{equation}

Since $\calR$ is an almost 
disjoint family, (\ref{eq:coding}) implies that 
for every $A' \in [\calA_\alpha^2 \union \calR]^{< \omega}$ and 
$r \in \calR \setminus A'$, also 
$\card{E(r) \setminus \bigcup A'} = \card{O(r) \setminus \bigcup A'} = 
\aleph_0$, as otherwise $r \intersection r'$ is infinite for 
some $r' \in A' \intersection \calR$.

If $\alpha < \omega_2$ is a successor ordinal, let 
$\dot{\mathbb{Q}}_\alpha = \dot{\mathbb{Q}}_\alpha^0$ be a 
$\mathbb{P}_\alpha$-name
for a proper poset of cardinality $\aleph_1$ such that
$\forces_{\mathbb{P}_\alpha} $``$\dot{\mathbb{Q}}_\alpha^0$ strongly
preserves the tightness of $\calA_1$''.
For limit $\alpha \in \omega_2$, 
unless explicitly mentioned otherwise, $\dot{\mathbb{Q}}_\alpha$ is a
$\mathbb{P}_\alpha$-name for the trivial poset.
Suppose
$F(\alpha)$ is a sequence
$\langle \dot{x_i} : i \in \omega \rangle$ of $\mathbb{P}_\alpha$-names 
such that $x_i = \dot{x_i}^{G_\alpha}$ is an infinite subset of 
$\omega$ such that for all $i \in \omega$, $x_i \in \mathcal{I}(
\calA^2_\alpha)^+$. 
The assumption $(\ast)$ and the almost disjointness of 
$\calR$ imply the existence of a limit ordinal 
$\eta_\alpha \in \omega_1$ with the property that there exist no finite subsets
$J, E$ of $\omega \cdot 2 \times (\omega_1 \setminus \eta_\alpha),
\calA^2_\alpha$, respectively, and $i \in \omega$ such that 
$x_i \subseteq \bigcup_{\langle \eta, \xi \rangle \in J} R_{\langle
\eta, \xi \rangle} \union \bigcup E$.
Fix such an $\eta_\alpha \in \omega_1$ as well as
$Z_\alpha \subset \omega$
coding a surjection of $\omega$ onto $\eta_\alpha$.
The following is the original definition from
\cite[Section 3]{Friedmanzdomskyy}.

\begin{definition}{(Friedman, Zdomskyy; \cite{Friedmanzdomskyy})} \label{def fz}
The partial order $\mathbb{K}_\alpha$ consists 
of all pairs 
$p = \langle \langle s, s^* \rangle, \langle c_k, y_k : k \in \omega 
\rangle \rangle$, such that: 

\begin{enumerate}[(1)]
\item $c_k \subseteq \omega_1 \setminus \eta_\alpha$ is closed 
bounded such that $S_{\alpha + k} \intersection c_k= 
\emptyset$; 
\item $y_k \from \card{y_k} \to 2$ is a function, where $\card{y_k} \in \omega_1$ is a limit, such that
\begin{itemize}
\item $\card{y_k} > \eta_\alpha$, $y_k \restriction \eta_\alpha = 0$; 
\item for all $\gamma < \card{y_k}$, $y_k(\eta_\alpha + 2\gamma) = 1$ if and only
if $\gamma = 0$ or 
$\gamma \in X_\alpha$, so that
$\mathrm{Even}(y_k) = (\set{\eta_\alpha}
\union (\eta_\alpha + X_\alpha))\intersection \card{y_k}$.\footnote{Recall $X_\alpha$
denotes the sufficiently absolute code for $\alpha$ given 
by Fact \ref{sufficiently absolute codes exist}.}
\end{itemize}

\item $s \in [\omega]^{< \omega}$ and $s^*$ is a finite 
subset of the set $\set{R_{\langle m, \xi \rangle}}[ m \in \Delta(s), \xi
\in c_m] \union \set{R_{\langle \omega + m, \xi \rangle}}[ m 
\in \Delta(s), \ y_m(\xi) = 1] \union \calA^2_\alpha$.
 Additionally, for all $n \in \omega$ such that 
$2n < \card{s \intersection R_{\langle 0,0 \rangle}}$,
$n \in Z_\alpha$ if and only if there exists $m \in \omega$ such that  
$(s \intersection R_{\langle 0, 0 \rangle})(2n) = 
R_{\langle 0 , 0 \rangle}(2m)$; 

\item  For all $k \in  C(s)$ and all limit 
$\gamma \in \omega_1$ such that 
$\eta_\alpha < \gamma \leq \card{y_k}$,
if $\gamma$ is a limit point of $c_k$ and $\gamma = \omega_1^\calM$ 
for some countable suitable model $\calM$ containing both
$y_k \restriction \gamma$ and $c_k \intersection \gamma$ as elements, 
then the following holds in $\calM$: 
``$[\mathrm{Even}(y_k) - \min (\mathrm{Even}(y_k))] 
\intersection \gamma$ 
is the code of some limit ordinal 
$\overline{\alpha} \in \omega_2$ such that 
$S_{\overline{\alpha} + k}$ is nonstationary.''
\end{enumerate}

\medskip
For $p = \langle \langle s, s^* \rangle, \langle c_k, y_k : k \in 
\omega \rangle \rangle$ and $q = \langle \langle t, t^* \rangle, \langle
d_k, z_k : k \in \omega \rangle \rangle$  
in $ \mathbb{K}_\alpha$, define $q$ to extend  
$p$ and write $q \leq p$ if and only if: 
\begin{enumerate}
\item $t$ end-extends $s$, $t^* \supseteq s^*$, and for all 
$x \in s^*$, $(t \setminus s) \intersection x = \emptyset$; 
\item For all $k \in C(t)$, $d_k$ end-extends $c_k$ and 
$z_k \supseteq y_k$. 
\end{enumerate}

For $p =\langle \langle s, s^* \rangle, \langle c_k, y_k : k \in 
\omega \rangle \rangle \in \mathbb{K}_\alpha$, let $\mathrm{Fin}(p) = 
\langle s, s^* \rangle$ denote the finite part of $p$, 
and let $\mathrm{Inf}(p) = \langle c_k, y_k : k \in 
\omega \rangle $ denote the infinite part of $p$. When 
$p,q \in \mathbb{K}_\alpha$ and $q \leq p$, we say $q$ 
is a \emph{pure extension of }$p$ if $\mathrm{Fin}(p) = 
\mathrm{Fin}(q)$.
\end{definition}

Proposition \ref{main} will give the properness of $\mathbb{K}_\alpha$, as well as the preservation of 
$\calA_1$.
%
For the proof however, we will need some intermediary lemmas:

\begin{lemma}[{\cite[Lemma 1]{FF2010}}] \label{extend1} For every $p = \langle \langle s, s^* \rangle, 
\langle c_k, y_k : k \in \omega \rangle \rangle$ in 
$\mathbb{K}_\alpha$ and every 
$\gamma \in \omega_1$, there exists a pure extension $q \leq p$ 
with $\mathrm{Inf}(q) = \langle d_k, z_k : k \in \omega \rangle$, such that $\card{z_k} \geq \gamma$ and $\max(d_k) \geq \gamma$, for
every $k \in \omega$. 
\end{lemma}

The following notion appears implicitly in 
\cite{Friedmanzdomskyy}.

\begin{definition} \label{defpreprocessed}
For a condition
 $p = \langle \langle s, s^* \rangle, \langle c_k, y_k : k \in 
\omega \rangle \rangle$ in $\mathbb{K}_\alpha$ and open dense $D \subseteq 
\mathbb{K}_\alpha$,
 we say $p$ is \emph{preprocessed} for $D$ if and only if for
every extension $q= \langle \langle t, t^* \rangle, 
\langle d_k, z_k : k \in \omega \rangle \rangle  \leq p$, if $q \in D$,
then already there is some 
$t_2^*$ such that  $q' = \langle \langle t, t_2^* \rangle, 
\langle c_k, y_k : k \in \omega \rangle \rangle$ is a condition 
in $\mathbb{K}_\alpha$ extending $p$, and $q' \in D$.
\end{definition}

Note that if $p \in \mathbb{K}_\alpha$ is preprocessed for
$D$ and $r \leq p$, then $r$ is preprocessed for $D$.

\begin{lemma}[{\cite[Claim 9]{Friedmanzdomskyy}}] \label{preprocessed}
 For any $p \in \mathbb{K}_\alpha$ and open dense 
$D \subseteq \mathbb{K}_\alpha$, there exists a pure extension
$q \leq p$ such that $q$
is preprocessed for $D$. 
\end{lemma} 


\begin{lemma} \label{pure outer hull}
Let $q \in \mathbb{K}_\alpha \intersection M$, where 
$M \prec H_\theta$ is a countable elementary submodel containing 
$\mathbb{K}_\alpha$ and $\calA_1$, and let 
$\dot{Z} \in M$ be a $\mathbb{K}_\alpha$-name for an element of 
$\calI(\calA_1)^+$.  Then 
$$W = \set{m \in \omega}[\exists p \leq q \ (\mathrm{Fin}(p) = 
\mathrm{Fin}(q) \wedge p \forces m\ \in \dot{Z})]$$
is an element of $\calI(\calA_1)^+ \intersection M$. 
\end{lemma}

\begin{proof}
Fix a finite $F \subseteq \calA_1 \intersection M$.
By assumption on $\dot{Z}$ we have that
$q \forces_{\mathbb{K}_\alpha} \dot{Z} \setminus
\bigcup F  \in 
\infsubsets$.

\begin{lemma}
For all $q \in \mathbb{K}_\alpha$ and $\dot{X}$ a 
$\mathbb{K}_\alpha$-name for an infinite subset of $\omega$, 
 there exists 
$p \leq q$ with $\mathrm{Fin}(p)= \mathrm{Fin}(q)$ and 
there exists $m^p \in \omega$ such that 
$p \forces m^p = \dot{X}(j)$. 
\end{lemma}

\begin{proof}
Fix $q = \langle \langle s, s^* \rangle, \langle c_k, y_k : k \in \omega 
\rangle \rangle$, $\dot{X}$ as above.
Consider the countable support product
$\mathbb{P}_s := \prod_{k \in C(s)} Q^{\eta_\alpha}(S_{\alpha+ k})
\times \calL^{\eta_\alpha}(Y_{\alpha+k})$, 
where
 $Q^{\eta_\alpha}(S_{\alpha+k})$ consists of closed bounded subsets $c_k \subseteq \omega_1 
\setminus \eta_\alpha$ such that 
$c_k \intersection S_{\alpha + k } = \emptyset$ and is 
ordered by end extension; the partial 
order 
$\calL^{\eta_\alpha}(Y_{\alpha+k})$ consists of 
functions 
$y_k \from \card{y_k} \to 2$ with domain $\card{y_k}$, such that 
\begin{itemize}
\item $\card{y_k} \in \omega_1 \setminus \eta_\alpha$ is a countable
limit ordinal and $y_k \restriction \eta_\alpha = 0$; 
\item $\mathrm{Even}(y_k) = 
(\set{\eta_\alpha} \union (\eta_\alpha + 
X_\alpha)) \intersection \card{y_k}$;
\item for all $\gamma \leq \card{y_k}$, if 
$\gamma = \omega_1^\calM$ for a suitable $\calM$ with 
$y_k \restriction \gamma \in \calM$ and $\gamma$ is a limit 
point of $c_k$, then 
$\calM \models $``$\mathrm{E ven}(y_k)$ is the code for 
some $\overline{\alpha} \in \omega_1$ such that 
$S_{\overline{\alpha} + k}$ is nonstationary ''. 
\end{itemize}
$\calL^{\eta_\alpha}(Y_{\alpha+k})$ is ordered by end-extension.
For notational simplicity we will suppress the superscript 
$\eta_\alpha$ in what follows. The ordering on $\mathbb{P}_s$ is given by 
$\langle d_k , z_k : k \in \omega \rangle 
\leq \langle c_k, y_k : k \in \omega \rangle$
if and only if $d_k$ is an end-extension of $c_k$ and 
$z_k \supseteq y_k$, i.e., if and only if 
$(d_k, z_k) \leq_{Q(S_{\alpha+k}) \times \calL(Y_{\alpha+k})} 
(c_k, y_k)$. Here, $\leq_{Q(S_{\alpha+k})}$ denotes the ordering 
of the club shooting forcing of Definition \ref{def club shooting}, however with the modification that 
condition are closed bounded subsets containing only 
ordinals strictly greater than $\eta_\alpha$. 
Likewise, $\leq_{\calL(Y_{\alpha+k})}$ denotes
the extension relation for the localization forcing of 
Definition \ref{localization forcing def}, 
where the $\Sigma_1$-formula being localized is 
$\varphi(\omega_1, X_\alpha)$ which asserts 
``$\mathrm{Even}(y_k) - \min(\mathrm{Even}(y_k))$ is the code for 
some $\overline{\alpha} \in \omega_1$ such that 
$S_{\overline{\alpha} + k}$ is nonstationary''.
Note that if $\langle c_k, y_k : k \in \omega \rangle 
\forces_{\mathbb{P}_s} \dot{X}(j)  = m_j$, then 
$\langle \langle \emptyset, \emptyset
\rangle, \langle c_k, y_k : k \in \omega \rangle \rangle 
\forces_{\mathbb{K}_\alpha} \dot{X}(j) = m_j$. Moreover
$\mathbb{P}_s$ is a complete suborder of 
$\mathbb{P}_{\emptyset}$.

For all $k \in C(s)$, find 
$c_k', y_k'$ such that $c_k' \leq_{Q(S_{\alpha+k})} c_k$
and $y_k' \leq_{\calL(Y_{\alpha + k})} y_k$, and 
$(c_k, y_k) \forces_{Q(S_{\alpha+ k}) \times 
\calL(Y_{\alpha+k})} \dot{X}(j) = \check{m}_j$
for some $m_j \in \omega$. 
Then $\langle c_k', y_k' : k \in \omega \rangle \in 
\mathbb{P}_s$ is an extension of 
$\langle c_k, y_k : k \in \omega \rangle$ and forces 
$\dot{X}(j) = \check{m}_j$. 
Therefore $\langle \mathrm{Fin}(q), \langle c_k', y_k' : k \in 
\omega \rangle \rangle \leq_{\mathbb{K}_\alpha} q$ and 
decides $\dot{X}(j)$.
\end{proof}

Now, letting 
$\dot{X}$ be a $\mathbb{K}_\alpha$-name for 
$\dot{Z} \setminus \bigcup F$, 
for every $k \in \omega$ there exists $j > k$ and 
a pure extension
$q_j \leq q$ in $\mathbb{K}_\alpha$,  
and there exists $m_j \in \omega$ with 
$q \forces \dot{X}(j) = \check{m}_j$. 
Then $$Y_k = \set{m_j}[j > k, q_j \leq q \wedge q_j \forces 
\check{m}_j = \dot{X}(j)]$$ is an infinite set witnessing $W \setminus \bigcup F$ is infinite.  
\end{proof}


\begin{proposition}\label{main}
For every $p = 
\langle \langle s, s^* \rangle, \langle c_k, y_k : k \in \omega \rangle
\rangle \in \mathbb{K}_\alpha$, every $\theta$ sufficiently large 
and countable elementary submodel $M \prec H_\theta$ containing
$p, \mathbb{K}_\alpha, \calA_1$, and every $B \in \mathcal{I}(\calA_1)$ such 
that $B \intersection Y$ is infinite for all $Y \in \mathcal{I}
(\calA_1)^+ \intersection M$, if $M \intersection \omega_1 = j
\not\in \bigcup_{k \in C(s)} S_{\alpha+k}$, then there is 
an $(M, \mathbb{K}_\alpha, \calA_1, B)$-generic condition $q \leq p$
such that $\mathrm{Fin}(q) = \mathrm{Fin}(p)$. 
\end{proposition}
\begin{proof}
Let $\theta$ be a sufficiently large regular cardinal and let
$M \prec H_{\theta}$ be a countable elementary submodel containing 
$p, \mathbb{K}_\alpha$ and $\calA_1$, such that 
$j = M \intersection \omega_1 \not\in 
\bigcup_{k \in C(s)} S_{\alpha + k}$. Fix $B \in \calI(\calA_1)$
such that $B \intersection Y$ is infinite for all 
$Y \in \calI(\calA_1)^+ \intersection M$. 
Let $\set{D_n}[n \in \omega]$ enumerate all open dense subsets of 
$\mathbb{K}_\alpha$ in $M$, and let $\set{\dot{Z_n}}[n \in \omega]$ enumerate
all $\mathbb{K}_\alpha$-names for subsets of $\omega$ in $M$ which
are forced to be in $\mathcal{I}(\calA_1)^+$ such that each name appears
infinitely often. Let $\langle j_n : n \in \omega \rangle$ 
be an increasing cofinal sequence of ordinals converging to $j$. 
We inductively define a descending sequence $\langle q_n : n \in \omega \rangle
\subseteq M \intersection \mathbb{K}_\alpha$ 
where $q_n = \langle \langle s, s^* \rangle, \langle d_k^n, 
c_k^n : k \in \omega \rangle \rangle$ and such that: 
\begin{enumerate}
\item $d_k^0 = c_k$, $z_k^0 = y_k$; 
\item For all $n \in \omega$ and $k \in C(s)$, 
$d_k^{n+1}$ is an end-extension of $d_k^n$ and
$z_k^{n+1} \supseteq z_k^n$; 
\item $\max(d_k^{n+1}), \card{z_k^{n+1}} \geq j_n$; 
\item $q_{n+1}$ is preprocessed for $D_n$;
\item $q_{n+1} \forces (\dot{Z}_n \intersection B) \setminus
n \neq \emptyset$.
\end{enumerate}

Assume $q_n$ has been constructed. 
First extend $q_n$ with a pure extension $q_n'$ such that 
$q_n$ is preprocessed for $D_n$. Next let
$$W_{n+1} = \set{m \in \omega}[\exists r \leq q_n' 
(\mathrm{Fin}(r)  =
\mathrm{Fin}(q) \wedge \ r \forces m \in \dot{Z}_n)].$$
Since $W_{n+1} \in \mathcal{I}(\calA_1)^+$ by Lemma \ref{pure outer hull},
 fix 
$m_n \in \omega$ such that $m_n > n$ and $m_n \in 
W_{n+1} \intersection B$. 
Let $r \leq q_n'$ be given by 
$m_n \in W$, and let $q_{n+1} \leq r$ 
be a pure extension of $r$ such that 
$\max(d_k^{n+1}) \geq j_n$ and $\card{z_k^{n+1}} \geq j_n$. 
Then $q_{n+1}$ satisfies the above  clauses, so this completes the 
inductive construction.

Set $d_k = \bigcup_{n \in \omega}d_k^n \union \set{j}$ and
$z_k = \bigcup_{n \in \omega}z_k^n$ for all $k \in C(s)$, 
and for $k \not\in C(s)$, let $d_k = z_k = \emptyset$. 
Define $q:=  \langle \langle s, s^* \rangle, \langle d_k, z_k : k \in \omega 
\rangle \rangle$. Then $q\in \mathbb{K}_\alpha$, which can be 
verified as in the proof of Lemma \ref{preprocessed}.
It remains to see that 
$q$ is an $(M, \mathbb{K}_\alpha, \calA_1, B)$-generic
condition.

First we show $q$ is $(M, \mathbb{K}_\alpha)$-generic by showing 
that for all $n \in \omega$, $D_n \intersection M$ is predense 
below $q$. Fix $n$ and let $r = \langle \langle t, t^* \rangle,
\langle d_k', z_k' : k \in \omega \rangle \rangle \leq q$, 
and we can assume $r \in D_n$. 
Then as $r \leq q_{n+1}$ and $q_{n+1}$ is preprocessed for $D_n$,
there is $r' = \langle \langle t, t_2^* \rangle, \langle
d_k^{n+1}, z_k^{n+1} : k \in \omega \rangle \leq q_{n+1}$ for 
some finite $t_2^* \in M$, such that already $r' \in D_n$. 
Clearly $r' \in M$. Then $r$ and $r'$ are compatible, as witnessed
by the condition $\langle \langle t, t^* \union t_2^* \rangle,
\langle d_k', z_k' : k \in \omega \rangle \rangle$. 

Lastly, for all $n \in \omega$ we have
$q \forces \card{\dot{Z}_n \intersection B} = \omega$. Let $\ell \in 
\omega$, and take any $r \leq q$. Find $i > \ell$
such that $\dot{Z}_i = \dot{Z}_n$; then 
$r \leq q_{i+1}$ and so since $q_{i+1}$ satisifes 
property (5), we have $$r \forces \emptyset 
\neq (\dot{Z}_i \intersection B) \setminus
i = (\dot{Z}_n \intersection B) \setminus i \subseteq 
(\dot{Z}_n \intersection B) \setminus \ell.$$
As $\ell$ was arbitrary this shows $q$ forces 
$\dot{Z}_n \intersection B$ is infinite, and therefore $q$ is an
$(M, \mathbb{K}_\alpha, \calA_1, B)$-generic condition.
\end{proof}

Let $H_\alpha$ be $\mathbb{K}_\alpha$-generic over 
$V[G_\alpha]$, and set
$Y_k^\alpha = \bigcup_{p \in H_\alpha} y_k$, 
$C_k^\alpha = \bigcup_{p \in H_\alpha} c_k$,
$a_\alpha = \bigcup_{p \in H_\alpha} s$, 
$\calA^2_{\alpha+1} = \calA_2 \union \set{a_\alpha}$. 
The following lemma gives consequences of forcing with 
$\mathbb{K}_\alpha$.

\begin{lemma}[{\cite[Claim 11]{Friedmanzdomskyy}}] \label{lem7} 
The following hold. 
\begin{enumerate}[(1)] 
    \item $a_\alpha \in [\omega]^\omega$ is almost disjoint from 
    all elements of $\calA_\alpha^2$;
    \item For all $i \in \omega$, $a_\alpha \intersection 
    x_i $ is infinite;
    \item For all $m \in \Delta(a_\alpha)$,
    $C_m^\alpha$ is a club in $\omega_1$ such that 
    $C_m^\alpha \intersection S_{\alpha+m} = \emptyset$, and for all
    $\xi \in \omega_1$,
    $\xi \in C_m^\alpha$ if and only if $a_\alpha \intersection 
    R_{\langle m, \xi \rangle}$ is finite;

    \item For all $m \in \Delta(a_\alpha)$,
    $Y_m^\alpha \from \omega_1 \to 2$ is a total function,
    and for all $\xi \in \omega_1$,
    $Y_m^\alpha(\xi) = 1$ if and only if 
    $a_\alpha \intersection R_{\langle \omega + m , \xi \rangle}$
    is finite;

   \item For all $n \in \omega$, $n \in Z_\alpha$ if and only if 
    there exists $m \in \omega$ such that 
    $(a_\alpha \intersection R_{\langle 0,0 \rangle })(2n) = 
    R_{\langle 0,0 \rangle}(2m)$; 

 \item For every $r \in \mathcal{R}$ and finite $A' \subseteq 
    \calA_{\alpha+1}^2$, $|E(r) \setminus 
    \bigcup A'| = | O(r) \setminus \bigcup A'| 
    = \omega$. 
    
\end{enumerate}

\end{lemma}

\begin{lemma}[{\cite[Corollary 12]{Friedmanzdomskyy}}] \label{cor8} 
    $\mathbb{K}_\alpha$ is $S_{-1}$-proper.
    Moreover, for every $ p = 
    \langle \langle s, s^* \rangle : \langle c_k,y_k : k \in \omega \rangle
     \rangle  \in 
    \mathbb{K}_\alpha$,
    the subposet
    $\mathbb{K}_\alpha \restriction p = 
    \set{r \in \mathbb{K}_\alpha}[ r \leq p]$ 
    is $(\omega_1 \setminus \bigcup_{n \in C(s)} S_{\alpha+n})$-
    proper.
\end{lemma}

This completes the definition 
of $\langle \mathbb{P}_\alpha, \dot{\mathbb{Q}}_\beta : 
\alpha \leq \omega_2, \beta < \omega_2 \rangle$. 
 
\begin{corollary} \label{full iteration of FZ is proper and preserves}
$\mathbb{P}_{\omega_2}$ is $S_{-1}$-proper and strongly preserves
the tightness of $\calA_1$. Moreover
for all $m \in \omega \setminus \Delta(a_\alpha)$, 
$S_{\alpha + m}$ remains stationary in $L[G]$. 
\end{corollary}

\begin{proof}
By Proposition \ref{main}, Lemmas \ref{cor8} and \ref{preservation of S proper omega bounding}.
\end{proof}

\begin{lemma}[{\cite[Lemma 13]{Friedmanzdomskyy}}] \label{pi12
definability of a lemma}
If $G$ is $\mathbb{P}$-generic over $L$, then 
 in $L[G]$, $\calA_2$ is definable by the 
following $\Pi_2^1$ formula: 
\begin{align*}
a \in \calA_2 \Leftrightarrow &    \ \forall \calM [(\calM \text{ is 
a countable suitable model, } a \in \calM) \\
& \exists \overline{\alpha}< \omega_2^\calM 
( \calM \models \forall m \in \Delta(a) (S_{\overline{\alpha} + m} 
\text{ is nonstationary})]
\end{align*}
\end{lemma}


With this we are ready to prove the main theorem of this section:

\medskip
\noindent
\emph{Proof of Theorem \ref{Theorem FZ section 2 sizes}}
Let $\mathbb{P}$ be the countable support iteration defined above,
let $G$ be $\mathbb{P}$-generic over $L$,
and let $\calA_2 = \set{a_\alpha }[\alpha < \omega_2]$,
where $a_\alpha = \dot{a}_\alpha^G$ where $\dot{a}_\alpha$ 
is the generic real added by $\mathbb{Q}_\alpha$. 
By Lemma \ref{lem7}, item (1), 
$\mathcal{A}_2$ is an almost disjoint family
of infinite subsets of $\omega$. To see 
it is tight, suppose 
that there is $\set{ x_i}[ i \in \omega ] \in 
L[G]$ such that $x_i \in \calI(\calA_2)^+$ for every 
$i \in \omega$. Then there is $\alpha < \omega_2$ such 
that $\langle x_i : i \in \omega \rangle \in 
L[G_\alpha]$, where $G_\alpha = G \intersection \mathbb{P}_\alpha$, so there is a sequence of 
$\mathbb{P}_\alpha$-names 
$\langle \dot{x}_i : i \in \omega \rangle \in L_{\omega_2}$
such that $x_i = \dot{x}_i^{G_\alpha}$. 
Since $F^{-1}(\langle \dot{x}_i : i \in \omega \rangle)$
is unbounded in $\omega_2$, 
there exists $\beta \geq \alpha$ such that 
$F(\beta) = \langle \dot{x}_i : i \in \omega \rangle$. 
By definition of $\dot{\mathbb{Q}}_\beta$, and 
Lemma \ref{lem7} item (2),
$a_\beta \intersection x_i$ is infinite in 
$L[G_\beta]$ for all $i \in \omega$, where 
$a_\beta$ is the $\mathbb{Q}_\beta$-generic real. As $a_\beta \in 
\calA_2$, we have $(\calA_2 $ is tight$)^{L[G]}$. 
That $\calA_2$ is $\Pi_2^1$-definable in 
$L[G]$ is by Lemma \ref{pi12
definability of a lemma}.

As $\calA_1 \in L$ is $\Pi_1^1$-definable 
in $L$, by Shoenfield absoluteness $\calA_1$ remains 
$\Pi_1^1$-definable in $L[G]$. 
By Proposition \ref{main}, for every $\alpha < \omega_2$, 
$\dot{\mathbb{Q}}_\alpha$ is a $\mathbb{P}_\alpha$-name
for a proper forcing strongly preserving the tightness 
of $\calA_1$. Thus 
$(\mathfrak{a} = \card{\calA_1} = \aleph_1 < \mathfrak{c} = \aleph_2 = 
\card{\calA_2})^{L[G]}$.
\qed

\section{Constellations and Definability} \label{section putting it all together}

 
With the above results, we can now prove
our main theorem. 
\begin{theorem}\label{themaintheorem|}

It is consistent that 
$\aleph_1 = \mathfrak{a} < \mathfrak{s} = \aleph_2$, 
there exists a $\Delta_3^1$ wellorder of the reals, 
as well as tight mad families of cardinality $\aleph_1$ and 
$\aleph_2$, which are respectively 
$\Pi_1^1$ and $\Pi_2^1$ definable.
\end{theorem}

\begin{proof}

We work in a model of $V=L$. The following can be obtained 
analagously as in Lemma \ref{statcostat sequence}, using
Solovay's theorem on the existence of $\omega_1$-many 
pairwise disjoint stationary subsets of $\omega_1$ and 
the relativized $\lozenge$ sequence 
$\lozenge_{T}$ for stationary $T \subseteq \omega_1$;
the assertion $\lozenge_{T}$ holds under $V=L$ for any
such $T$. 

\begin{lemma}
There are pairwise disjoint stationary subsets $\{T_i\}_{i\in 3}$
of $\omega_1$ such that for each $i \in 2$ there is a  
sequence 
$\vec{S}^i = \langle S_\alpha^i \subseteq T_i : \alpha < \omega_2 \rangle$
of stationary costationary subsets of $\omega_1$, 
such that for distinct $\alpha, \beta < \omega_2$, 
$S_\alpha^i \intersection S_\beta^i$ is bounded. 
Moreover, whenever $\calM, \calN$ are suitable models such 
that $\omega_1^\calM = \omega_1^\calN$, then 
$\langle (S_\alpha^i)^\calM : \alpha < \omega_2^\calM \intersection 
\omega_2^\calN \rangle = 
\langle (S_\alpha^i)^\calN : \alpha < \omega_2^\calM \intersection 
\omega_2^\calN \rangle$. 
\end{lemma} 

 
Let $\calA_1$ be a coanalytic tight mad family, 
and fix $\Sigma_1$-definable bookeeping function 
$F \from \mathrm{Lim} \intersection \omega_2 \to 
    L_{\omega_2}$, and a $\Sigma_1$-definable almost disjoint family 
    $\mathcal{R}$ such that $F, \calR$ are
    as in the proof of Corollary \ref{delta 13 with a < s}
    and Definition \ref{def fz}.
 Define a countable support iteration
 $\langle \mathbb{P}_\alpha , \dot{\mathbb{Q}}_\beta : 
 \alpha \leq \omega_2, \beta < \omega_2 \rangle$, with
 $\mathbb{P}_0$ being the trivial forcing. Let $G_\alpha$
 be $\mathbb{P}_\alpha$-generic over $L$; the wellordering 
 $<_\alpha$ on the reals of $L[G_\alpha]$ is defined
 exactly as for the proof of Theorem \ref{delta 13 with a < s}. 
 Let $\dot{\mathbb{Q}}_\alpha$ be a $\mathbb{P}_\alpha$-name
 for the two-step iteration $\dot{\mathbb{Q}}_\alpha^0 \ast
 \dot{\mathbb{Q}}_\alpha^1$ such that 
 $\dot{\mathbb{Q}}_\alpha^0$ is a 
 $\mathbb{P}_\alpha$-name for $\mathbb{Q}$ of Definition 
 \ref{def creature}. Unless one of the following cases
 occurs, $\dot{\mathbb{Q}}_\alpha^1$ is a 
 $\mathbb{P}_\alpha \ast \dot{\mathbb{Q}}_\alpha^0$-name
 for the trivial forcing. 
 
 \underline{Case I:} $\alpha$ is a limit ordinal 
 and $F(\alpha) = \set{\dot{b}_i}[i \in \omega]$
 is a sequence of $\mathbb{P}_\alpha \ast \dot{\mathbb{Q}}_\alpha^0$-names such that 
 in $V[G_\alpha]$, $\dot{b}_i^G$ is an element of 
 $\calI(\calA_\alpha^2)^+$ for each $i \in \omega$. 
 Then let $\dot{\mathbb{Q}}_\alpha^1$ be a 
 $\mathbb{P}_\alpha \ast \dot{\mathbb{Q}}_\alpha^0$ name for the forcing notion 
 $\mathbb{K}_\alpha$ of Definition \ref{def fz}, with 
 respect to the same countable limit ordinal $\eta_\alpha \in 
 \omega_1$, and modifying item (1) by letting 
 $c_k \subseteq \omega_1 \setminus {\eta_\alpha}$ be a closed bounded
 subset such that $S_{\alpha + k}^0 \intersection c_k = 
 \emptyset$.  
 
 \underline{Case II:} $F(\alpha) = \set{ \sigma_x^\alpha, 
 \sigma_y^\alpha}$ is a pair of $\mathbb{P}_\alpha$-names 
 for reals in $L[G_\alpha]$ such that 
 $\sigma_x^\alpha <_L \sigma_y^\alpha$ (i.e., 
 $x = (\sigma_x^\alpha)^G <_\alpha y = (\sigma_y^\alpha)^G$). 
 In this case define $\dot{\mathbb{Q}}_\alpha^1$ to 
 be a $\mathbb{P}_\alpha \ast \dot{\mathbb{Q}}_\alpha^0$-name
 for $\mathbb{Q}_\alpha^1 = 
 \mathbb{K}_\alpha^0 \ast \mathbb{K}_\alpha^1 \ast 
 \mathbb{K}_\alpha^2$ as defined in the proof of 
 Corollary \ref{delta 13 with a < s}, 
 however modifying the definition of $\mathbb{K}_\alpha^0$
 by taking closed bounded subsets of $\omega_1 \setminus
 S^1_{\alpha+k}$, for $k \in \Delta(x_\alpha \ast y_\alpha)$. 

This completes the definition of $\mathbb{P} = 
\mathbb{P}_{\omega_2}$. Note that for all
$\alpha < \omega_2$, $\dot{\mathbb{Q}}_\alpha$ is a 
$\mathbb{P}_\alpha$-name for either a proper, 
a $T_1 \union T_2$-proper, or a 
$T_0 \union T_2$-proper forcing notion, and in each 
case $\dot{\mathbb{Q}}_\alpha$ strongly preserves the
tightness of $\calA_1$.  Therefore
$\mathbb{P}$ is $T_2$-proper and so preserves 
$\omega_1$ as well as the tightness of $\calA_1$. 
Let $G$ be $\mathbb{P}$-generic over $L$.
 For each $\alpha < \omega_2$, $\mathbb{Q}_\alpha = 
\mathbb{Q}$ adds a real 
not split by the ground model reals, so
$(\mathfrak{s} = \aleph_2)^{L[G]}$. For cofinally many $\alpha < \omega_2$, $\mathbb{Q}_\alpha$ 
adds an infinite $a_\alpha \subseteq \omega$ such that 
$\calA_2 = \set{a_\alpha}[\alpha < \omega_2]$ is a tight mad family with a $\Pi_2^1$ definition in $L[G]$. Moreover the wellorder
$<_G = \bigcup_{\alpha < \omega_2}$ can be shown to be
a $\Delta_3^1$ wellorder of the reals of $L[G]$, as 
in Theorem \ref{delta 13 with a < s}. Then altogether we have  that
$L[G]$ witnesses
the conclusions of the theorem. 
 \end{proof}

\section{Concluding remarks and Questions}

The constellation $\aleph_1= \mathfrak{b} < \mathfrak{a} = 
\mathfrak{s}$, first established in \cite{Shelah84}, is shown 
to be consistent with a $\Delta_3^1$ wellorder of the reals 
in \cite{FF2010}. The consistency of 
$\mathfrak{a} = \mathfrak{c}$ with a $\Delta_3^1$ wellorder
of the reals and a tight projective witness of $\mathfrak{a}$ is 
established in \cite{Friedmanzdomskyy}. 
In our results 
$\mathfrak{a}$ stays small, with a definable witness. 
Of interest however remains:

\begin{question}
    Is $\mathfrak{b} < \mathfrak{a}$ 
    consistent with a $\Pi_2^1$ (tight) mad family? 
\end{question}

A model of $\aleph_2 < \mathfrak{b} = \mathfrak{c}$ with 
a $\Pi_1^1$  witness to $\mathfrak{a}$ is due to Brendle and Khomskii \cite{BK12}.
This construction however relies heavily on the preservation of splitting 
families and thus of interest remains:

\begin{question}
Is it consistent with
$\mathfrak{a} < \mathfrak{c}$ or even with 
$\mathfrak{a} < \mathfrak{s} = \mathfrak{c}$
that there exist
coanalytic mad families of sizes 
$\mathfrak{a}$ and $\mathfrak{c}$? 
\end{question}

Note that our results heavily depend on the use of countable support iterations
and so our techniques can only yield models with
  $\mathfrak{c} = \aleph_2$. 
  A natural question is:
\begin{question}\label{3 sizes}
Is it consistent that $\card{\mathrm{spec}(\mathfrak{a})} 
\geq 3$ and for each $\kappa \in \mathrm{spec}(\mathfrak{a})$
there exists a projective mad family of 
size $\kappa$ with an optimal definition? 
\end{question}

\bibliographystyle{amsalpha}
\bibliography{bibtristan}

\providecommand{\bysame}{\leavevmode\hbox to3em{\hrulefill}\thinspace}
\providecommand{\MR}{\relax\ifhmode\unskip\space\fi MR }
\providecommand{\MRhref}[2]{%
  \href{http://www.ams.org/mathscinet-getitem?mr=#1}{#2}
}
\providecommand{\href}[2]{#2}
\begin{thebibliography}{BHST22}

\bibitem[Abr10]{Abraham2010}
Uri Abraham, \emph{Proper forcing}, Handbook of {S}et {T}heory. {V}ols. 1, 2, 3 (Matthew Foreman and Akihiro Kanamori, eds.), Springer, Dordrecht, 2010, pp.~333--394.

\bibitem[BFB22]{bergfalk2022projective}
Jeffrey Bergfalk, Vera Fischer, and Corey {Bacal Switzer}, \emph{Projective well orders and coanalytic witnesses}, Annals of Pure and Applied Logic \textbf{173} (2022), no.~8, 103135.

\bibitem[BHK76]{BHK76}
James~E. Baumgartner, Leo~A. Harrington, and Eugene~M. Kleinberg, \emph{Adding a closed unbounded set}, The Journal of Symbolic Logic \textbf{41} (1976), no.~2, 481--482.

\bibitem[BHST22]{Bakke_Haga_2021}
Karen Bakke~Haga, David Schrittesser, and Asger T\"ornquist, \emph{Maximal almost disjoint families, determinacy, and forcing}, Journal of Mathematical Logic \textbf{22} (2022), no.~1, Paper No. 2150026, 42.

\bibitem[BK13]{BK12}
J\"org Brendle and Yurii Khomskii, \emph{Mad {F}amilies {C}onstructed from {P}erfect {A}lmost {D}isjoint {F}amilies}, The Journal of Symbolic Logic \textbf{78} (2013), no.~4, 1164--1180.

\bibitem[Bla93]{Blass94}
Andreas Blass, \emph{Simple cardinal characteristics of the continuum}, Set theory of the reals ({R}amat {G}an, 1991), Israel Mathematical Conference Proceedings, vol.~6, Bar-Ilan Univ., Ramat Gan, 1993, pp.~63--90.

\bibitem[BPS80]{BPS80}
Bohuslav Balcar, Jan Pelant, and Petr Simon, \emph{The space of ultrafilters on $\mathbb{N}$ covered by nowhere dense sets}, Fundamenta Mathematicae \textbf{110} (1980), no.~1, 11--24 (en).

\bibitem[Bri24]{Brian24}
Will Brian, \emph{Partitioning the real line into {B}orel sets}, The Journal of Symbolic Logic \textbf{89} (2024), no.~2, 549–568.

\bibitem[BY05]{BY05}
J\"org Brendle and Shunsuke Yatabe, \emph{Forcing indestructibility of {MAD} families}, Annals of Pure and Applied Logic \textbf{132} (2005), 271--312.

\bibitem[CF11]{CF2011}
Andr\'es~Eduardo Caicedo and Sy-David Friedman, \emph{{BPFA} and {P}rojective {W}ell-orderings of the {R}eals}, The Journal of Symbolic Logic \textbf{76} (2011), no.~4, 1126--1136.

\bibitem[Cum10]{Cummings2010}
James Cummings, \emph{Iterated forcing and elementary embeddings}, Handbook of set theory. {V}ols. 1, 2, 3 (Matthew Foreman and Akihiro Kanamori, eds.), Springer, Dordrecht, 2010, pp.~775--883.

\bibitem[Dav82]{RDavid}
Ren\'e David, \emph{A very absolute {$\Pi_2^1$} real singleton}, Annals of Mathematical Logic (1982), no.~3 (2-3), 101--120.

\bibitem[Dev17]{Devlin}
Keith~J. Devlin, \emph{Constructibility}, Perspectives in Logic, Cambridge University Press, 2017.

\bibitem[Dow95]{dow95}
Alan Dow, \emph{More set-theory for topologists}, Topology and its Applications \textbf{64} (1995), no.~3, 243--300.

\bibitem[FF10]{FF2010}
Vera Fischer and Sy~David Friedman, \emph{Cardinal characteristics and projective wellorders}, Annals of Pure and Applied Logic \textbf{161} (2010), no.~7, 916--922.

\bibitem[FFK13]{FFK13}
Vera Fischer, Sy~David Friedman, and Yurii Khomskii, \emph{Co-analytic {M}ad families and {D}efinable {W}ellorders}, Archive for Mathematical Logic \textbf{52} (2013), no.~7-8, 809--822.

\bibitem[FFST25]{FFST}
Vera Fischer, Sy~David Friedman, David Schrittesser, and Asger T\"ornquist, \emph{Good projective witnesses}, Annals of Pure Applied Logic \textbf{176} (2025), no.~8, Paper No. 103606, 29.

\bibitem[FFZ11]{FFZ11}
Vera Fischer, Sy-David Friedman, and Lyubomyr Zdomskyy, \emph{Projective wellorders and mad families with large continuum}, Annals of Pure and Applied Logic \textbf{162} (2011), 853--862.

\bibitem[Fis08]{Fischerthesis}
Vera Fischer, \emph{The consistency of arbitrarily large spread between the bounding and the splitting numbers}, Phd thesis, York university, Toronto, Ontario, 2008.

\bibitem[FS19]{Fischershelah19}
Vera Fischer and Saharon Shelah, \emph{The spectrum of independence}, Archive for Mathematical Logic \textbf{58} (2019), no.~7-8, 877--884.

\bibitem[FS21]{FSmed}
Vera Fischer and David Schrittesser, \emph{A {S}acks indestructible co-analytic maximal eventually different family}, Fundamenta Mathematicae \textbf{252} (2021), no.~2, 179--201.

\bibitem[FS22]{Fischershelah22}
Vera Fischer and Saharon Shelah, \emph{The spectrum of independence, {II}}, Annals of Pure and Applied Logic \textbf{173} (2022), no.~9, Paper No. 103161, 9.

\bibitem[FS25]{FS25}
Vera Fischer and Lukas Schembecker, \emph{Partitions of the {B}aire space into compact sets}, Fundamenta Mathematicae \textbf{269} (2025), no.~1, 45--69.

\bibitem[FSS25]{FSS25}
Vera Fischer, L.~Schembecker, and David Schrittesser, \emph{Tight cofinitary groups}, Annals of Pure and Applied Logic \textbf{176} (2025), 103570.

\bibitem[FZ10]{Friedmanzdomskyy}
Sy-David Friedman and Lyubomyr Zdomskyy, \emph{Projective mad families}, Annals of Pure and Applied Logic \textbf{161} (2010), no.~12, 1581--1587.

\bibitem[GHT20]{GHT}
Osvaldo Guzman, Michael Hru\v{s}\'ak, and Osvaldo Tellez, \emph{Restricted {MAD} families}, The Journal of Symbolic Logic \textbf{85} (2020), no.~1, 149–165.

\bibitem[G{\"o}d39]{Godel39}
Kurt G{\"o}del, \emph{Consistency-proof for the generalized continuum-hypothesis}, Proceedings of the National Academy of Sciences \textbf{25} (1939), no.~4, 220--224.

\bibitem[Gol93]{Goldsterntools}
Martin Goldstern, \emph{Tools for your forcing construction}, Set theory of the reals ({R}amat {G}an, 1991), Israel Mathematical Conference Proceedings, vol.~6, Bar-Ilan Univ., Ramat Gan, 1993, pp.~305--360.

\bibitem[Gol98]{Goldsterntaste}
\bysame, \emph{A {T}aste of {P}roper {F}orcing}, Set Theory (Dordrecht) (Carlos~Augusto Di~Prisco, Jean~A. Larson, Joan Bagaria, and A.~R.~D. Mathias, eds.), Springer Netherlands, 1998, pp.~71--82.

\bibitem[Har77]{Harrington77}
Leo Harrington, \emph{Long projective wellorderings}, Annals of Mathematical Logic \textbf{12} (1977), no.~1, 1--24.

\bibitem[Hec72]{Hechler72}
Stephen~H. Hechler, \emph{Short complete nested sequences in {$\beta N\backslash N$} and small maximal almost-disjoint families}, General Topology and its Applications \textbf{2} (1972), 139--149.

\bibitem[HGF03]{HrusakFerreira2003}
Michael Hrušák and Salvador García-Ferreira, \emph{Ordering {M}ad {F}amilies a {L}a {K}atětov}, The Journal of Symbolic Logic \textbf{68} (2003), no.~4, 1337--1353.

\bibitem[Hru01]{HrusakRationals}
Michael Hrušák, \emph{{MAD} families and the rationals}, Commentationes Mathematicae Universitatis Carolinae \textbf{42} (2001), no.~2, 345--352.

\bibitem[Jec97]{Jech2nd}
Thomas Jech, \emph{Set {T}heory}, 2nd ed., Perspectives in Mathematical Logic, Springer-Verlag, Berlin, Heidelberg, 1997.

\bibitem[Jec03]{Jech}
\bysame, \emph{Set theory: The third millennium edition, revised and expanded}, 3rd ed., Springer Monographs in Mathematics, Springer-Verlag, Berlin, Heidelberg, 2003.

\bibitem[Kur01]{Kurilic}
Miloš~S. Kurilić, \emph{Cohen-{S}table families of {S}ubsets of {I}ntegers}, The Journal of Symbolic Logic \textbf{66} (2001), no.~1, 257--270.

\bibitem[Mal89]{Malykhin}
V.~I. Malykhin, \emph{Topological properties of {C}ohen generic extensions}, Trudy Moskov. Mat. Obshch. \textbf{52} (1989), 3--33, 247.

\bibitem[Mat77]{MATHIAS197759}
Adrian R.~D. Mathias, \emph{Happy families}, Annals of Mathematical Logic \textbf{12} (1977), no.~1, 59--111.

\bibitem[Mil89]{miller1989}
Arnold~W. Miller, \emph{Infinite combinatorics and definability}, Annals of Pure and Applied Logic \textbf{41} (1989), no.~2, 179--203.

\bibitem[NN18]{NN18}
Itay Neeman and Zach Norwood, \emph{Happy and mad families in $l(\mathbb{R})$}, The Journal of Symbolic Logic \textbf{83} (2018), 572--597.

\bibitem[Rag09]{Raghavan}
Dilip Raghavan, \emph{Maximal almost disjoint families of functions}, Fundamenta Mathematicae \textbf{204} (2009), no.~3, 241--282.

\bibitem[RS99]{RScreatures}
Andrzej Ros{\l}anowski and Saharon Shelah, \emph{Norms on {P}ossibilities. {I}: {F}orcing with {T}rees and {C}reatures}, Memoirs of the American Mathematical Society, American Mathematical Society, 1999.

\bibitem[She82]{Shelahproperforcing}
Saharon Shelah, \emph{Proper forcing}, Lecture Notes in Mathematics, vol. 940, Springer-Verlag, Berlin-New York, 1982.

\bibitem[She84]{Shelah84}
\bysame, \emph{On cardinal invariants of the continuum}, Axiomatic set theory ({B}oulder, {C}olo., 1983), Contemporary Mathematics, vol.~31, American Mathematical Society, Providence, RI, 1984, pp.~183--207.

\bibitem[She17]{ShelahPIP}
\bysame, \emph{Proper and improper forcing}, 2 ed., Perspectives in Logic, Cambridge University Press, 2017.

\bibitem[SJ70]{solovayjensen}
Robert~B. Solovay and Ronald~M. Jensen, \emph{Some applications of almost disjoint sets}, Mathematical Logic and Formulations of Set Theory (1970), 84--104.

\bibitem[SS15]{Shelahspinas}
Saharon Shelah and Otmar Spinas, \emph{{MAD} spectra}, Journal of Symbolic Logic (2015), no.~80, 243--262.

\bibitem[Tö18]{Tornquistdefinability}
Asger Törnquist, \emph{Definability and almost disjoint families}, Advances in Mathematics \textbf{330} (2018), 61--73.

\end{thebibliography}

\end{document}